\documentclass[12pt]{amsart}

\usepackage{amsmath, amsthm, amssymb, amsfonts, amscd, graphicx, endnotes, tikz, color, hyperref}

\usepackage{mathtools}

\usepackage[margin=01.2in]{geometry}
\usepackage[nodisplayskipstretch]{setspace}
\usetikzlibrary{arrows}

\def\PP{{\mathbb P}}

\def\ZZ{{\mathbb Z}}

\def\0{{\mathbf 0}}
\def\1{{\mathbf 1}}

\def\Ucal{{\mathcal U}}

\def\Kbar{\bar{K}}

\def\Ksep{K^\mathrm{sep}}

\def\arith{\mathrm{arith}}
\def\geom{\mathrm{geom}}
\def\ldim{\underline{\mathrm{dim}}}
\def\udim{\overline{\mathrm{dim}}}

\def\Aut{\mathrm{Aut}}
\def\ch{\mathrm{char}}

\def\ker{\mathrm{ker}\,}

\def\ab{\mathrm{ab}}
\def\Gal{\mathrm{Gal}}

\def\com{\mathrm{com}}

\def\PGL{\mathrm{PGL}}

\def\GL{\mathrm{GL}}

\theoremstyle{plain}

\newtheorem{thm}{Theorem}

\newtheorem{conj}{Conjecture}
\newtheorem{cor}[thm]{Corollary}
\newtheorem{prop}[thm]{Proposition}
\newtheorem{lem}[thm]{Lemma}

\theoremstyle{definition}

\title[]{The Minkowski dimension of the image \\ of an arboreal Galois representation}

\author{Chifan Leung}

\address{Chifan Leung \newline Beijing International Center for Mathematical Research \newline Peking University \newline Yiheyuan Road \newline Beijing China, 100871 \newline chifan1728@pku.edu.cn}

\author{Clayton Petsche}

\address{Clayton Petsche \newline Department of Mathematics \newline Oregon State University \newline Corvallis, Oregon, U.S.A. \newline petschec@oregonstate.edu}

\date{October 2, 2026}
\keywords{Arithmetic dynamics, arboreal Galois representations, iterated monodromy groups, Minkowski dimension, rooted tree automorphisms. 2020 Mathematics Subject Classification: 37P05 37P15  11R32 20E08}

\begin{document}

\begin{abstract}
We consider the Minkowski dimension of the arboreal Galois group $G_{f,\alpha}$ associated to a rational map $f:\PP^1\to\PP^1$ and a base point $\alpha\in\PP^1(K)$.  This is a subgroup of the automorphism group of the infinite $d$-ary rooted tree whose vertices are indexed by the backward orbit $f^{-\infty}(\alpha)$.  We show that the Minkowski dimension exists for the profinite iterated monodromy groups $G_f^\arith$ and $G_f^\geom$, and that these two groups have the same dimension.  We prove a dichotomy theorem stating that $G_f^\arith$ and $G_f^\geom$ are either the full tree automorphism group or else have non-maximal dimension.  We identify several cases of interest in which dimension non-maximality $\udim(G_{f,\alpha})<1$ holds, including the cases of postcritical base point, the case of periodic base point, the case in which $f$ is a nontrivial iterate, and the postcritically finite case.  We identify several cases of interest in which dimension minimality $\dim(G_{f,\alpha})=0$ holds, including the power, Chebyshev, Latt\`es, and abelian cases.  We formulate a conjecture on dimension minimality, which if true would imply a conjecture of Andrews-Petsche on abelian arboreal Galois groups.
\end{abstract}

\maketitle

\tableofcontents

\newpage



\section{Introduction}\label{IntroSect}

Let $f:\PP^1\to\PP^1$ be a rational map of degree $d\geq2$ defined over a field $K$ of characteristic zero or characteristic $>d$, and let $\alpha\in\PP^1(K)$.  Under these conditions, we call $(f,\alpha)$ an {\em arboreal pair} of degree $d\geq2$ defined over $K$, and we call $\alpha$ the {\em base point}.  Let $\Ksep\subseteq\Kbar$ be separable and algebraic closures of $K$.  

Denote by $f^n=f\circ\dots\circ f$ the composition of $n$ iterations of $f$.  For each $n\geq0$, define the $n$-th inverse image set associated to the arboreal pair $(f,\alpha)$ by
\begin{equation*}
f^{-n}(\alpha) = \{\beta\in\PP^1(\Kbar)\mid f^n(\beta)=\alpha\}.
\end{equation*}
Let $f^{-\infty}(\alpha)=\cup_{n\geq0}f^{-n}(\alpha)$ denote the full backward $f$-orbit of $\alpha$.  

Each arboreal pair $(f,\alpha)$ determines an infinite $d$-ary rooted tree $T_{f,\alpha}$ whose vertex set is the disjoint union $\amalg_{n=0}^{\infty} L_n$, where for each $n\geq0$ we index the set $L_n=\{v_\beta\}$ by the distinct points in $\beta\in f^{-n}(\alpha)$.  The only edge relations in $T_{f,\alpha}$ occur between vertices $v_{\beta'}\in L_{n+1}$ and $v_{\beta}\in L_n$ for some $n\geq0$, and these two vertices share an edge if and only if $f(\beta')=\beta$.

A point $\alpha\in\PP^1(K)$ is {\em postcritical} for $f$ if $f^n(\gamma)=\alpha$ for some critical point $\gamma\in\PP^1(\Kbar)$ and some $n\geq1$.  If the base point $\alpha$ is {\em not} postcritical, then the preimage tree $T_{f,\alpha}$ is a {\em complete} infinite rooted $d$-ary tree, meaning that every vertex in each $L_n$ meets exactly $d$ distinct vertices in $L_{n+1}$.  When $\alpha$ is postcritical, we have $|f^{-n}(\alpha)|<d^n$ for all large enough $n$, and thus $T_{f,\alpha}$ is not complete as a $d$-ary rooted tree. Many authors in arboreal Galois theory exclude the case of postcritical base point, but it suits our purposes to allow for this possibility.

Because the characteristic of $K$ is either zero or $>d$, it follows that $K(f^{-\infty}(\alpha))/K$ is a separable (hence Galois) extension.  As $f$ and $\alpha$ are defined over $K$, each automorphism $\sigma\in\Gal(K(f^{-\infty}(\alpha))/K)$ fixes $\alpha$ and commutes with the rational map $f$, from which it follows that $\Gal(K(f^{-\infty}(\alpha))/K)$ acts on $T_{f,\alpha}$ by rooted tree automorphisms.  The resulting injective homomorphism 
\[
\rho_{f,\alpha}:\Gal(K(f^{-\infty}(\alpha))/K)\to\Aut(T_{f,\alpha})
\]
is known as the arboreal Galois representation associated to $f$ and $\alpha$.  We denote by 
\[
G_{f,\alpha}=\rho_{f,\alpha}(\Gal(K(f^{-\infty}(\alpha))/K))  
\]
the image in $\Aut(T_{f,\alpha})$ of the representation.  Both $\rho_{f,\alpha}$ and $G_{f,\alpha}$ depend on the choice of base field $K$ over which $f$ and $\alpha$ are defined, but we suppress this dependence in the notation.

Using a natural metric to realize the profinite topology on $\Aut(T_{f,\alpha})$ (this metric is defined precisely in $\S$\ref{TreeSection}), one may measure the size or complexity of any closed subgroup $G$ of $\Aut(T_{f,\alpha})$ using the {\em lower and upper Minkowski dimensions}, which are defined by
\begin{equation}\label{MinkowskiDimsGeneral}
\ldim(G) =\liminf_{\epsilon\to0}\frac{\log{N_\epsilon(G)}}{\log(1/\epsilon)} \hskip1cm
\udim(G) =\limsup_{\epsilon\to0}\frac{\log{N_\epsilon(G)}}{\log(1/\epsilon)} 
\end{equation}
where $N_\epsilon(G)$ is the smallest number of $\epsilon$-balls in $\Aut(T_{f,\alpha})$ needed to cover $G$.  It always holds that 
\[
0\leq \ldim(G)\leq \udim(G)\leq1.
\]  
If equality $\ldim(G)= \udim(G)$ holds, we say that the Minkowski dimension $\dim(G)$ exists, and we define it to be this common value 
\begin{equation*}
\dim(G):=\ldim(G)=\udim(G)=\lim_{\epsilon\to0}\frac{\log{N_\epsilon(G)}}{\log(1/\epsilon)}.
\end{equation*}

The lower Minkowski dimension $\ldim(G)$ coincides with its Hausdorff dimension, which has been considered by several authors in the arithmetic dynamics literature, beginning with Boston-Jones \cite{MR2520459}, Pink \cite{PinkQRM,pink2}, and others, e.g. \cite{MR4430121,MR3736808,MR3220023}.  It is always true that $\dim_{\mathrm{Haus}}\leq\ldim$ in any bounded metric space (see \cite{MR3616046} $\S$ 1.1-1.2).  But in our situation of a closed subgroup $G$ of the automorphism group of a $d$-ary rooted tree, equality $\dim_{\mathrm{Haus}}(G)=\ldim(G)$ always holds (see \cite{MR1422889} Thm. 2.4).  Thus lower Minkowski dimension means the same thing as Hausdorff dimension throughout this paper.  

Examples of rooted tree automorphism groups $G$ with $\ldim(G)<\udim(G)$ can be constructed.  These examples are somewhat contrived, and we conjecture that this cannot occur naturally in the case of arboreal Galois groups $G_{f,\alpha}$.  

\begin{conj}\label{MinkowskiDimExistsStrngSense}
Let $f:\PP^1\to\PP^1$ be a rational map of degree $d\geq2$ defined over a field $K$ of characteristic zero or $>d$, and let $\alpha\in\PP^1(K)$.  Then the Minkowski dimension $\dim(G_{f,\alpha})$ exists.
\end{conj}

But in the absence of a general proof of the existence of the limit $\dim(G_{f,\alpha})$, it will be desirable for our purposes to distinguish between lower and upper Minkowski dimensions.  

One rationale for Conjecture \ref{MinkowskiDimExistsStrngSense} is that the Minkowski dimension does exist for the profinite iterated monodromy groups $G_f^\arith$ and $G_f^\geom$ associated to $f$, as we show in Theorem \ref{IMGDimExistsEqualThmIntro}.  Developed and utilized by Boston-Jones \cite{MR2520459}, Pink \cite{PinkQRM,pink2}, Jones \cite{MR3220023} and others, these groups are defined as follows.  Replacing $K$ with the field $K(t)$ of rational functions in a variable $t$ over $K$, instead of considering a base point $\alpha\in K$, one chooses as base point the variable $t$ itself.  Since $f$ is defined over $K$, its postcritical locus is confined to $\PP^1(\Kbar)$.  Therefore $t$ is not postcritical and $T_{f,t}$ is a complete infinite $d$-ary rooted tree.   We obtain the injective arboreal Galois representation
\[
\rho_{f}:\Gal(K(f^{-\infty}(t))/K(t))\to\Aut(T_{f,t})
\]
defined over the function field $K(t)$, and we may consider its image
\[
G_{f}^\arith=\rho_{f}(\Gal(K(f^{-\infty}(t))/K(t))),
\]
called the {\em arithmetic profinite iterated monodromy group} associated to $f$.  Replacing $K(t)$ with $\Ksep(t)$, one may also define the arboreal representation 
\[
\overline{\rho}_{f}:\Gal(\Ksep(f^{-\infty}(t))/\Ksep(t))\to\Aut(T_{f,t})
\]
and its image
\[
G_{f}^\geom=\overline{\rho}_{f}(\Gal(\Ksep(f^{-\infty}(t))/\Ksep(t))),
\]
called the {\em geometric profinite iterated monodromy group} associated to $f$.  Since $\overline{\rho}_f$ factors through $\rho_f$ we have an inclusion 
\begin{equation}\label{GeomArithInclusion}
G_{f}^\geom\subseteq G_{f}^\arith\subseteq \Aut(T_{f,t}).
\end{equation}

The most well-understood examples occur in the case of quadratic rational maps, due to work of Pink \cite{PinkQRM,pink2}.  It was shown by Pink in the $d=2$ case that $G_{f}^\geom=G_{f}^\arith=\Aut(T_{f,t})$ unless either (i) the rational map is postcritically finite; or (ii) a critical orbit collision of the form $f^{r+1}(\gamma_1)=f^{r+1}(\gamma_2)$ holds for some $r\geq1$, where $\gamma_1,\gamma_2\in\PP^1(\Ksep)$ are the two critical points of $f$.  In each case (i) and (ii), it holds that $\dim(G_{f}^\arith)<1$, with explicit calculation of the dimension in many cases of interest.

For each $\alpha\in\PP^1(K)$ which is not $f$-postcritical, an argument specializing $t=\alpha$ provides an embedding 
\begin{equation}\label{SpecializationIntro}
G_{f,\alpha}\subseteq G_{f}^\arith\subseteq\Aut(T_{f,t}).
\end{equation}
In some favorable situations over a number field $K$, it can be shown that 
\[
G_{f,\alpha}= G_{f}^\arith=\Aut(T_{f,t})
\]
or at least that $G_{f,\alpha}$ has finite index as a subgroup of $\Aut(T_{f,t})$ (see e.g. \cite{MR3937585,MR4057534,MR3937588}); in these cases it follows at once that $\dim(G_{f,\alpha})=1$.  There are also known examples of calculations in the intermediate range $0<\dim(G_{f,\alpha})<1$.  For example, Benedetto-Faber-Hutz-Juul-Yasufuku \cite{MR3736808} have calculated $\dim(G_{f,\alpha})=1-\frac{\log2}{3\log 6}=0.871...$ for the post critically finite cubic map $f(x)=-2x^3+3x^2$ and a certain infinite family of $\alpha\in\PP^1(K)$.  For the basilica map $f(x)=x^2-1$, it follows from the main result of Ahmad-Benedetto-Cain-Carroll-Fang \cite{MR4430121} that $\dim(G_{f,\alpha})=2/3$ for infinitely many $\alpha\in\PP^1(K)$.

Our first result shows that the Minkowski dimensions of $G_f^\geom$ and $G_f^\arith$ exist, and in fact these two dimensions always coincide.

\begin{thm}\label{IMGDimExistsEqualThmIntro}
Let $f:\PP^1\to\PP^1$ be a rational map of degree $d\geq2$ defined over a field $K$ of characteristic zero or $>d$.  The Minkowski dimensions of $G_f^\geom$ and $G_f^\arith$ both exist and $\dim(G_f^\geom)=\dim(G_f^\arith)$.
\end{thm}

The following is a dichotomy theorem, which states that $G_f^\geom$ and $G_f^\arith$ must either both coincide with the full tree automorphism group, or else both must have non-maximal Minkowski dimension.

\begin{thm}\label{IMGDichotomyThmIntro}
Let $f:\PP^1\to\PP^1$ be a rational map of degree $d\geq2$ defined over a field $K$ of characteristic zero or $>d$.  Then one and only one of the following two conditions holds: either
\begin{itemize}
\item[{\bf (i)}]   $G_{f}^\geom=G_{f}^\arith=\Aut(T_{f,t})$; or 
\item[{\bf (ii)}]   $\dim(G_{f}^\geom)=\dim(G_{f}^\arith)<1$. 
\end{itemize}
\end{thm}

The common value of $\dim(G_{f}^\geom)$ and $\dim(G_{f}^\arith)$ serves as a uniform upper bound for upper Minkowski dimension in the sense that
\begin{equation}\label{MDIMGUpperBound}
\udim(G_{f,\alpha})\leq \dim(G_{f}^\arith) \hskip1cm \text{for all }\alpha\in\PP^1(K).
\end{equation}
See $\S$\ref{IMGSect} for the proof of (\ref{MDIMGUpperBound}).

An application of (\ref{MDIMGUpperBound}) is the following.  Recall that $f:\PP^1\to\PP^1$ is said to be {\em postcritically finite} if every critical point $\gamma\in\PP^1(\Kbar)$ of $f$ is $f$-preperiodic.  It is known in the postcritically finite case that $G_f^\geom$ is topologically finitely generated (see \cite{MR3220023} $\S$3.1), which makes case {\bf (i)} in Theorem \ref{IMGDichotomyThmIntro} impossible.  This proves the following corollary.

\begin{cor}\label{PCFCorIntro}
Suppose that $f:\PP^1\to\PP^1$ is postcritically finite.  Then 
\[
\dim(G_{f}^\geom)=\dim(G_{f}^\arith)<1.  
\]
In particular, $\udim(G_{f,\alpha})\leq \dim(G_{f}^\arith)<1$ for all $\alpha\in\PP^1(K)$.
\end{cor}

Assuming that $f$ is postcritically finite and that $\alpha\in\PP^1(K)$ is not $f$-postcritical, it follows from (\ref{FiniteIndexImpliesDimMaxIntro}) and Corollary \ref{PCFCorIntro} that $G_{f,\alpha}$ has infinite index as a subgroup of $\Aut(T_{f,\alpha})$; this gives an alternate proof of a theorem of Jones (\cite{MR3220023} Thm. 3.1).

\bigskip

For general arboreal Galois groups $G_{f,\alpha}$, we are interested in distinguishing between cases of minimal dimension $\dim(G_{f,\alpha})=0$, cases of maximal dimension $\dim(G_{f,\alpha})=1$, and cases of intermediate dimension in which neither extreme holds.  

Beginning with the question of dimension maximality, a first observation is that, assuming that the base point $\alpha$ is not $f$-postcritical, we have $\dim(\Aut(T_{f,\alpha}))=1$ and therefore the implication
\begin{equation}\label{FiniteIndexImpliesDimMaxIntro}
[\Aut(T_{f,\alpha}):G_{f,\alpha}]<+\infty \hskip5mm \Rightarrow \hskip5mm \dim(G_{f,\alpha})=1
\end{equation}
holds (Proposition \ref{MinDimPropertiesProp}).  Therefore, in the case of non-postcritical base point $\alpha$, one may consider any dimension non-maximality result of the type $\udim(G_{f,\alpha})<1$ as a quantitative strengthening of the statement that $G_{f,\alpha}$ has infinite index as a subgroup of $\Aut(T_{f,\alpha})$.

We do not know whether the converse statement of (\ref{FiniteIndexImpliesDimMaxIntro}) is true; in other words, we do not know whether maximality $\dim(G_{f,\alpha})=1$ may occur {\em only} in the case that $G_{f,\alpha}$ has finite index as a subgroup of $\Aut(T_{f,\alpha})$.  This equivalence does hold for $G_f^\geom$ and $G_f^\arith$, as follows from Theorem \ref{IMGDichotomyThmIntro}.  And in the $d=2$ case we can prove the following conditional result: if Jones' finite index conjecture (\cite{MR3220023} Conjecture 3.11) is true, then the converse of (\ref{FiniteIndexImpliesDimMaxIntro}) holds for quadratic rational maps over number fields.  This is proved in $\S$\ref{MaxConjSect}.

While a complete general characterization of dimension maximality $\dim(G_{f,\alpha})=1$ seems currently out of reach, we are able to prove dimension non-maximality $\udim(G_{f,\alpha})<1$ in some cases of interest.  We have already seen that the postcritically finite case is one such instance; the following proposition summarizes three more.

\begin{thm}\label{NonMaximalityThmIntro}
Let $f:\PP^1\to\PP^1$ be a rational map of degree $d\geq2$ defined over a field $K$ of characteristic zero or $>d$, and let $\alpha\in\PP^1(K)$.  If any one or more of the following conditions hold:
\begin{itemize}
\item[{\bf (a)}] $\alpha$ is $f$-postcritical;
\item[{\bf (b)}] $\alpha$ is $f$-periodic; or
\item[{\bf (c)}] $f=g^m$ for some rational map $g$ and some integer $m\geq2$;
\end{itemize}
then $\udim(G_{f,\alpha})<1$.
\end{thm}

Under each of the conditions {\bf (a)}, {\bf (b)}, and {\bf (c)}, explicit upper bounds on $\udim(G_{f,\alpha})$ are given in terms of $f$ and $\alpha$; these are stated precisely in Propositions \ref{PostCritBasePointProp}, \ref{PeriodicBasePointProp}, and \ref{IterateDimProp}.  

\bigskip

We turn now to the question of dimension minimality $\dim(G_{f,\alpha})=0$.  A first observation is that dimension minimality $\dim(G_{f,\alpha})=0$ must occur for all $\alpha\in\PP^1(K)$ whenever $f$ is a power map, a Chebyshev map, or a Latt\`es map.     We make these calculations in Propositions \ref{PowerMapsAreSmall}, \ref{ChebyshevMapsAreSmall}, and \ref{LattesMapsAreSmall} respectively.

A second observation is that dimension minimality $\dim(G_{f,\alpha})=0$ is forced whenever $G_{f,\alpha}$ is abelian.  This follows from the following purely group-theoretic result.

\begin{thm}\label{AbelianSubgroupsAreSmallIntro}
Let $T$ be an infinite $d$-ary rooted tree and let $G$ be an abelian subgroup of $\Aut(T)$.  Then $\dim(G)=0$.
\end{thm}

Theorem \ref{AbelianSubgroupsAreSmallIntro} was proved in the complete, $p$-ary case for prime $p$ by Ab\'ert-Vir\'ag \cite{MR2114819}.  Our treatment generalizes their result to the possibly incomplete, $d$-ary case for any $d\geq2$.  We prove Theorem \ref{AbelianSubgroupsAreSmallIntro} in $\S$\ref{MainAbelianSection}.  

As a sample application of this result, consider the example $f(x)=2/(x-1)^2$, recently shown by Ejder-Go\l aska-Kara-Nienhaus-\"Ulkem \cite{EGKNU} to satisfy $\dim(G_{f}^\arith)=0$.  This vanishing of Minkowski dimension may be explained by Proposition \ref{LattesMapsAreSmall}, as $f$ is a Latt\`es map (it is conjugate to the Latt\`es family considered in Ex.~4.7 of \cite{MR4686319}).  Alternatively, it can also be shown that $G_{f}^\geom$ has an abelian subgroup $U$ of finite index (\cite{EGKNU} Lem. 3.5 and Prop. 3.7), and therefore using Theorem \ref{IMGDimExistsEqualThmIntro}, Theorem \ref{AbelianSubgroupsAreSmallIntro}, and Proposition \ref{MinDimPropertiesProp} we have
\[
\dim(G_f^\arith)=\dim(G_f^\geom)=\dim(U)=0,
\]
giving an alternative explanation for this result found in \cite{EGKNU}.

It would be interesting to classify all cases of dimension minimality $\dim(G_{f,\alpha})=0$.  In this direction we propose the following conjecture for number fields.

\begin{conj}\label{MinkowskiDimSmallConj}
Let $f:\PP^1\to\PP^1$ be a rational map defined over a number field $K$, and let $\alpha\in\PP^1(K)$ be any non-exceptional point for $f$.  If $\dim(G_{f,\alpha})=0$, then $f$ is $\Kbar$-conjugate to a power map, a Chebyshev map, or a Latt\`es map.  
\end{conj}

The statement of Conjecture \ref{MinkowskiDimSmallConj} is known to be true in the special case of $G_f^\geom$ and $G_f^\arith$ for quadratic polynomials; this follows from the work of Pink \cite{PinkQRM}, who has calculated $\dim(G_{f}^\geom)$ explicitly for all quadratic polynomials, and shown that the answer depends only the critical portrait of $f$.  It may be deduced from these calculations that $\dim(G_{f}^\arith)=\dim(G_{f}^\geom)>0$ for all quadratic polynomial maps $f$ except for those which are $\Kbar$-conjugate to either the squaring map $P(x)=x^2$ or the Chebyshev map $T_2(x)=x^2-2$.

As a result of Theorem \ref{AbelianSubgroupsAreSmallIntro}, we can show that Conjecture \ref{MinkowskiDimSmallConj} would have implications toward a classification of abelian arboreal Galois groups $G_{f,\alpha}$.  Let $K$ be a number field and let $K^\ab$ be its maximal abelian extension.  It has been conjectured by Andrews-Petsche \cite{MR4150256} (see also Ferraguti-Ostafe-Zannier \cite{MR4686319} for a corrected formulation of the conjecture) that, in the polynomial case, the group $G_{f,\alpha}$ can only be abelian in certain well understood situations.

\begin{conj}\label{AndrewsPetscheConj}(Andrews-Petsche \cite{MR4150256}, Ferraguti-Ostafe-Zannier \cite{MR4686319})
Let $K$ be a number field, let $f:\PP^1\to\PP^1$ be a polynomial map of degree $d\geq2$ defined over $K$, and let $\alpha\in \PP^1(K)$ be a non-exceptional point for $f$.  If $G_{f,\alpha}$ is abelian, then either:
\begin{itemize}
\item The pair $(f,\alpha)$ is $K^\ab$-conjugate to $(x^{d},\zeta)$ for a root of unity $\zeta$; or 
\item The pair $(f,\alpha)$ is $K^\ab$-conjugate to $(\pm T_d(x),\zeta+\frac{1}{\zeta})$ for a root of unity $\zeta$.
\end{itemize} 
\end{conj}

Until recently, Conjecture \ref{AndrewsPetscheConj} had only been proved in special cases (e.g. \cite{MR4150256, MR4686319, MR4216695,MR4914085,MR3611309}).  After submitting this paper for publication, we have learned that proofs of Conjecture \ref{AndrewsPetscheConj} have now been announced independently by Ferraguti-Pagano \cite{FP} and Ji-Song-Xie \cite{JSX}.

We can prove the following conditional result.

\begin{thm}\label{NewConjImpliesAP}
If Conjecture \ref{MinkowskiDimSmallConj} is true for polynomial maps, then Conjecture \ref{AndrewsPetscheConj} is true.
\end{thm}

The proof of Theorem \ref{NewConjImpliesAP} is an application of Theorem \ref{AbelianSubgroupsAreSmallIntro}, and is given in $\S$\ref{MinConjSect}.

\bigskip

{\bf Acknowledgements.}  The authors would like to thank Rob Benedetto, \"Ozlem Ejder, Rafe Jones, and Tom Tucker for helpful discussions.


\section{Automorphism groups of rooted trees}\label{TreeSection}

A {\em rooted tree} is a tree $T$ together with a choice of distinguished vertex called the {\em root} of $T$.  Let $w$ be a non-root vertex of $T$.  If $v$ is a vertex distinct from $w$, and $v$ lies in the unique path from $w$ to the root, we say that $w$ is a {\em descendant} of $v$, and that $v$ is an {\em ancestor} of $w$.  If $w$ is a descendant of $v$ and adjacent to $v$, we say that $w$ is the {\em child} of $v$ and $v$ is the {\em parent} of $w$.  Define the {\em level} of a vertex $v$ to be the number of edges in the unique path from $v$ to the root.  For each $k=0,1,2,\dots$, let $L_k$ denote the set of level $k$ vertices of $T$.

Let $d\geq2$ be an integer.  By an {\em infinite $d$-ary rooted tree} we mean a rooted tree $T$ such that each vertex has at least 1 and at most $d$ child vertices.  We say that an infinite $d$-ary rooted tree $T$ is {\em complete} if each vertex has exactly $d$ distinct child vertices.  Denote by $T^\com$ the complete infinite $d$-ary rooted tree.

By a {\em finite $d$-ary rooted tree of height $n$} we mean a rooted tree $T$ such that: (i) each vertex of level $k=0,1,2,\dots, n-1$ has at least 1 and at most $d$ child vertices, and (ii) each vertex of level $n$ has no child vertices.  We say that a finite $d$-ary rooted tree $T$ of height $n$ is {\em complete} if each vertex of level $k=0,1,2,\dots, n-1$ has exactly $d$ child vertices.  Denote by $T^\com_n$ the complete finite $d$-ary rooted tree of height $n$.

An {\em automorphism} $\sigma:T\to T$ of a $d$-ary rooted tree $T$ (finite or infinite) is a bijection on vertices of $T$ fixing the root vertex and preserving edge relations.  It follows that any such automorphism restricts to a bijection $L_k\to L_k$ of each level.  The group of all automorphisms of $T$ is denoted by $\Aut(T)$. 

If $T$ is a finite $d$-ary rooted tree of height $n$, a well-known recursive counting argument shows that
\begin{equation}\label{AutGroupCountBound}
|\Aut(T)|\leq d!^{(d^n-1)/(d-1)} \hskip1cm \text{(with equality if $T$ is complete).}
\end{equation}




Let $T$ be an infinite $d$-ary rooted tree.  For each $n\geq0$, let $T_n$ denote the finite subtree obtained by selecting only the vertices of $T$ of levels $k=0,1,2,\dots,n$, and the edges between them.  Thus $T_n$ is a finite $d$-ary rooted tree of height $n$.  Let 
\[
r_{n}:\Aut(T)\rightarrow\Aut(T_{n}) \hskip1cm r_{n}(\sigma) = \sigma|_{T_n}
\]
be the surjective group homomorphism obtained by restricting each automorphism $\sigma:T\rightarrow T$ to an automorphism $\sigma|_{T_{n}}:T_{n}\rightarrow T_{n}$.  

The profinite topology on the group $\Aut(T)$ can be realized by the metric defined for $\sigma,\tau\in\Aut(T)$ by
\begin{equation}
d(\sigma,\tau) = \inf\left\{\epsilon_n: n\geq0 \text{ and }r_n(\sigma)=r_n(\tau)\right\},
\end{equation}
where
\[
\epsilon_n=\frac{1}{|\Aut(T_n^\com)|}=\frac{1}{d!^{(d^n-1)/(d-1)}}.
\]
In other words, two automorphisms in $\Aut(T)$ are close to each other when they restrict to the same automorphism on the finite subtree $T_n$ for large $n$.  

In the complete case $T=T^\com$, the choice of $\epsilon_n=1/|\Aut(T_n^\com)|$ in the definition of the metric $d(\sigma,\tau)$ is traditional, as it ensures that the radius $\epsilon_n$ of the neighborhood $\ker(r_n)$ of the identity in $\Aut(T^\com)$ coincides with the reciprocal of its index as a subgroup of $\Aut(T^\com)$. This choice also forces $\dim(\Aut(T^\com))=1$; see Proposition \ref{FiniteIndexLargeImageProp}.

In the general case, we call special attention to the fact that our choice of distance in the definition of the metric $d(\sigma,\tau)$ is always the same quantity $\epsilon_n=1/|\Aut(T_n^\com)|$, even in the non-complete case $T\neq T^\com$.  This choice is convenient from the point of view of Minkowski dimension, as it allows for clean formulas such as Proposition \ref{MinDimAbstractCalcProp}, which holds regardless of whether or not $T$ is complete.  In the arboreal setting, this choice allows for the calculations of Proposition \ref{MainArbHausdorffDimCalc} to hold, treating the case of $f$-postcritical base point $\alpha$ on the same footing with the general case.  This choice of metric also has the consequence that $\udim(\Aut(T))<1$ whenever $T\neq T^\com$; see Proposition \ref{NonCompleteProp}. 

The following proposition calculates the lower and upper Minkowski dimensions of a subgroup $G$ of $\Aut(T)$ in terms of the images of $G$ under the surjective restriction homomorphisms $r_n:\Aut(T)\to\Aut(T_n)$.

\begin{prop}\label{MinDimAbstractCalcProp}
Let $T$ be an infinite $d$-ary rooted tree, and let $G$ be a closed subgroup of $\Aut(T)$.  Then
\begin{equation*}
\begin{split}
\ldim(G) & =\frac{d-1}{\log(d!)}\liminf_{n\rightarrow+\infty}\frac{1}{d^n}\log|r_n(G)| \\
\udim(G) & =\frac{d-1}{\log(d!)}\limsup_{n\rightarrow+\infty}\frac{1}{d^n}\log|r_n(G)| \\
\dim(G) & =\frac{d-1}{\log(d!)}\lim_{n\rightarrow+\infty}\frac{1}{d^n}\log|r_n(G)| \hskip1cm (\text{when this limit exists}).
\end{split}
\end{equation*}
Moreover, it holds that $0\leq \ldim(G)\leq \udim(G)\leq1$.
\end{prop}

\begin{proof}
For each possible distance $\epsilon_n=1/|\Aut(T_n^\com)|=d!^{-(d^{n}-1)/(d-1)}$ under the metric $d(\sigma,\tau)$, suppose that $N=|r_n(G)|$.  We may select a collection of $\sigma_i\in G$ with distinct $r_n(\sigma_i)$ and $r_n(G)=\{r_n(\sigma_i)\}_{i=1}^{N}$.  Then $\{B_{\epsilon_n}(\sigma_i)\}_{i=1}^N$ is a minimal covering of $G$ using balls of radius $\epsilon_n$, so by (\ref{MinkowskiDimsGeneral}) we have 
\begin{equation*}
\begin{split}
\udim(G) & =\limsup_{\epsilon\to0}\frac{\log N_\epsilon(G)}{\log(1/\epsilon)} \\
	& =\limsup_{n\to+\infty}\frac{\log N_{\epsilon_n}(G)}{\log(1/\epsilon_n)} \\
	& =\limsup_{n\to+\infty}\frac{\log|r_n(G)|}{\log(d!^{(d^{n}-1)/(d-1)})} \\
	& = \frac{d-1}{\log(d!)}\limsup_{n\rightarrow+\infty}\frac{1}{d^n}\log|r_n(G)|,
\end{split}
\end{equation*}
and similarly in the case of $\ldim(G)$.  

The bounds $0\leq \ldim(G)\leq \udim(G)$ are trivial.  To see that $\udim(G)\leq 1$, since $r_n(G)\subseteq\Aut(T_n)$, from (\ref{AutGroupCountBound}) we have
\begin{equation*}
|r_n(G)|\leq |\Aut(T_n)|\leq d!^{(d^n-1)/(d-1)}
\end{equation*}
and $\udim(G)\leq 1$ follows from this and the established identity for $\udim(G)$.
\end{proof}

The identities in Proposition \ref{MinDimAbstractCalcProp} are special cases of more general calculations made by Barnea-Shalev \cite{MR1422889}, following work of Abercrombie \cite{MR1281541}.  

\begin{prop}\label{MinDimPropertiesProp}
Let $T$ be an infinite $d$-ary rooted tree, and let $G\subseteq H$ be closed subgroups of $\Aut(T)$.  Then $\ldim(G)\leq \ldim(H)$ and $\udim(G)\leq \udim(H)$.  If $G$ has finite index in $H$, then $\ldim(G)=\ldim(H)$ and $\udim(G)=\udim(H)$.
\end{prop}

\begin{proof}
The assertions that $\ldim(G)\leq \ldim(H)$ and $\udim(G)\leq \udim(H)$ are evident from Proposition \ref{MinDimAbstractCalcProp}, since $r_n(G)\subseteq r_n(H)$ for all $n\geq0$.  Assume now that $G$ has finite index in $H$.  Elementary calculations show that the map
\[
\bar{r}_{n}:H/G\rightarrow r_n(H)/r_n(G) \hskip1cm \bar{r}_{n}(\sigma G)=r_{n}(\sigma)r_{n}(G)
\]
on coset spaces is well-defined and surjective, which implies that
\begin{equation}\label{CosetBound}
(H:G)\geq (r_n(H):r_n(G))
\end{equation}
for all $n\geq0$.  
Thus letting $M=(H:G)$, the bound (\ref{CosetBound}) implies that $|r_{n}(G)|\geq|r_n(H)|/M$ for all $n$ and so
\begin{equation*}
\begin{split}
\ldim(G) & = \frac{d-1}{\log(d!)}\liminf_{n\rightarrow+\infty}\frac{1}{d^n}\log|r_n(G)| \\
	& \geq \frac{d-1}{\log(d!)}\liminf_{n\rightarrow+\infty}\frac{1}{d^n}(\log|r_n(H)|-\log M) \\
	& = \frac{d-1}{\log(d!)}\liminf_{n\rightarrow+\infty}\frac{1}{d^n}\log|r_n(H)| =\ldim(H),
\end{split}
\end{equation*} 
which, combined with the opposite bound proved already, completes the proof of $\ldim(G)=\ldim(H)$.  The proof that $\udim(G)=\udim(H)$ is similar.
\end{proof}

\begin{prop}\label{FiniteIndexLargeImageProp}
If $G$ is a finite index subgroup of $\Aut(T^\com)$, then $\dim(G)=1$.  
\end{prop}

\begin{proof}
We have $|r_n(\Aut(T^\com))|=|\Aut(T_n^\com)|=d!^{(d^n-1)/(d-1)}$ by (\ref{AutGroupCountBound}), and hence $\dim(\Aut(T^\com))=1$ using Proposition \ref{MinDimAbstractCalcProp}.  It then follows from Proposition \ref{MinDimPropertiesProp} that $\dim(G)=1$.
\end{proof}

\begin{prop}\label{NonCompleteProp}
Let $T$ be an infinite $d$-ary rooted tree.  If $T$ is not complete, then $\udim(\Aut(T))<1$.
\end{prop}

\begin{proof}
The hypothesis $T\neq T^\com$ means there exists some $m\geq1$ for which $|L_m|<d^m$.  We will show that 
\begin{equation}\label{NonCompleteUDimBound}
\udim(\Aut(T))\leq \frac{|L_m|}{d^m}<1
\end{equation}
which will complete the proof.

Let $R=|L_m|$.  Since each vertex in $L_{n}$ shares an edge with at most $d$ vertices of $L_{n+1}$, we have $|L_{n+1}|\leq |L_n|d$ for all $n\geq0$.  In particular $|L_{m+1}|\leq Rd$, and repeating this argument we obtain $|L_{m+k}|\leq Rd^{k}$ for each $k\geq0$.  

For each $n\geq0$ let $r_{n+1,n}:\Aut(T_{n+1})\to\Aut(T_n)$ denote the surjective restriction map.  Each element of $\ker(r_{n+1,n})$ can only permute the $d$ vertices above any given vertex of $L_n$, so $|\ker(r_{n+1,n})|\leq d!^{|L_n|}$.  Therefore
\[
|\Aut(T_{n+1})|=|\ker(r_{n+1,n})||\Aut(T_{n})|\leq d!^{|L_n|}|\Aut(T_{n})|.
\]
For $k\geq0$ we therefore have 
\[
|\Aut(T_{m+k+1})|\leq d!^{|L_{m+k}|}|\Aut(T_{m+k})|\leq d!^{Rd^{k}}|\Aut(T_{m+k})|,
\]
which we can repeat recursively to arrive at
\[
|\Aut(T_{m+k})|\leq d!^{R(1+d+d^2+\dots+d^{k-1})}|\Aut(T_{m})|
\]
for each $k\geq1$.

Since $r_n(\Aut(T))=\Aut(T_n)$, we apply Proposition \ref{MinDimAbstractCalcProp} to obtain
\begin{equation*}
\begin{split}
\udim(\Aut(T)) & =\frac{d-1}{\log(d!)}\limsup_{k\rightarrow+\infty}\frac{1}{d^{m+k}}\log|r_{m+k}(\Aut(T))| \\
	& \leq\frac{d-1}{\log(d!)}\limsup_{k\rightarrow+\infty}\frac{1}{d^{m+k}}\log(d!^{R(1+d+d^2+\dots+d^{k-1})}|\Aut(T_{m})|) \\
	& =\frac{d-1}{\log(d!)}\limsup_{k\rightarrow+\infty}\frac{R(1+d+d^2+\dots+d^{k-1})}{d^{m+k}}\log(d!) \\
	& =\frac{R}{d^m},
\end{split}
\end{equation*}
completing the proof.
\end{proof}


\section{Abelian subgroups of $\Aut(T)$}\label{MainAbelianSection}

The goal of this section is to prove Theorem \ref{AbelianSubgroupsAreSmallIntro} on the vanishing Minkowski dimension of abelian subgroups of $\Aut(T)$.

Let $T$ be a finite $d$-ary rooted tree of height $n$, and let $v$ be a level $k$ vertex for $0\leq k\leq n-1$.  We say that $v$ is a {\em complete vertex} of $T$ if $v$ has $d$ child vertices, and $v$ is an {\em incomplete vertex} of $T$ if $v$ has fewer than $d$ child vertices.  (It is most convenient to consider the level $n$ vertices to be neither complete nor incomplete.)  Let $I(T)$ be the number of incomplete vertices of $T$.

Given a finite $d$-ary rooted tree $T$ of height $n$, let $G$ be a subgroup of $\Aut(T)$.  Recall that $T/G$ is the tree whose vertices are the $G$-orbits of vertices of $T$.  As the action of $G$ fixes the root of $T$ and preserves the edge relations of $T$, the quotient $T/G$ is itself a $d$-ary rooted tree of height $n$, with levels and edge relations induced by those of $T$.

The following lemma was proved by Ab\'ert-Vir\'ag  (\cite{MR2114819} Lemma 8.1)  in the case of complete $p$-ary rooted trees for prime $p$.  Our statement generalizes this to the possibly incomplete case and to arbitrary degree $d\geq2$.

\begin{lem}\label{MainAbelianSubgroupBound}
Let $T$ be a finite $d$-ary rooted tree of height $n\geq0$, and let $G$ be an abelian subgroup of $\Aut(T)$.  Let $q$ be the order of the largest abelian subgroup of $S_d$.  Then $|G|\leq q^{I(T/G)}$.
\end{lem}

\begin{proof}
If $n=0$ then $G$ is the trivial group and the result holds trivially.

Proceeding by induction, let $T$ be a finite $d$-ary rooted tree of height $n\geq1$, let $G$ be an abelian subgroup of $\Aut(T)$, and assume that the lemma is true for $d$-ary rooted trees of height $n-1$.  

Let $T^*$ be the height 1 rooted subtree of $T$ consisting only of levels 0 and 1 of $T$, and let $G^*$ be the group of automorphisms of $T^*$ obtained by restriction of automorphisms in $G$ to $T^*$.  Let $H$ be the kernel of the surjective restriction map $G\to G^*$, hence $H$ is the subgroup of $G$ consisting of all automorphisms fixing every level 1 vertex of $T$, and  
\begin{equation}\label{FITTreeAutAbelian}
|G|=|G^*||H|.
\end{equation}

We will show that 
\begin{equation}\label{R1BoundAbelian}
|G^*|\leq\begin{cases}
1 & \text{ if the root of $T/G$ is complete} \\
q & \text{ if the root of $T/G$ is not complete}.
\end{cases}
\end{equation}
If $G^*$ is the trivial group then (\ref{R1BoundAbelian}) is obvious.  If $G^*$ is not the trivial group then the level 1 vertices of $T$ contain a nontrivial $G$-orbit and therefore $T/G$ has $<d$ level 1 vertices.  This means that the root of $T/G$ is not complete.  Since $G^*$ is isomorphic to an abelian subgroup of $S_d$ we have $|G^*|\leq q$ and (\ref{R1BoundAbelian}) follows.

Next let $w_1,w_2,\dots,w_s$ be the a complete set of representatives for the distinct $G$-orbits of level 1 vertices of $T$.  For each $1\leq i\leq s$ let $\Lambda_i$ be the subtree of $T$ whose vertices are $w_i$ and all descendants of $w_i$, and whose edges are all edges of $T$ among these vertices. Thus each $\Lambda_i$ is itself a finite $d$-ary rooted tree of height $n-1$ with root $w_i$.  By construction, the $G$-orbit of any nonroot vertex $v\in T$ meets exactly one of the subtrees $\Lambda_1,\Lambda_2,\dots,\Lambda_s$, since the level 1 ancestor of $v$ contains one and only one of the $w_i$ in its $G$-orbit.

Let 
\begin{equation*}
\begin{split}
G_i & =\{\sigma\in G\mid \sigma(w_i)=w_i\} \\
K_i & =\{\sigma\in G\mid \sigma(v)=v\text{ for all }v\in \Lambda_i\} \\
H_i & =G_i/K_i.
\end{split}
\end{equation*}
Thus $G_i$ is the stabilizer of the vertex $w_i$, $K_i$ is the stabilizer of the entire descendent subtree $\Lambda_i$ of $w_i$.  Since $K_i$ is the kernel of the action of $G_i$ on $\Lambda_i$, the quotient $H_i$ acts faithfully on $\Lambda_i$.  

Consider the map 
\begin{equation}\label{InjectivePartitionMap}
H\hookrightarrow H_1\times H_2\times\dots\times H_s
\end{equation}
defined by inclusion $H\hookrightarrow G_i$ followed by the quotient map $G_i\to H_i$ at each factor.  To see that this map is injective, suppose that $\sigma\in H$ satisfies $\sigma\in K_i$ for each $i$; that is, $\sigma$ restricts to the identity automorphism of each subtree $\Lambda_i$.  If $v\in T$ is a level $m$ vertex for $1\leq m\leq n$, then as we have observed above, $v$ lies in the $G$-orbit of some vertex $v_i$ in one of the trees $\Lambda_1,\Lambda_2,\dots,\Lambda_s$; say $\tau(v_i)=v$ for some $\tau\in G$.  Using that $G$ is abelian and that $\sigma$ fixes $v_i$, we have 
\[
\sigma(v)=\sigma(\tau(v_i))=\tau(\sigma(v_i))=\tau(v_i)=v.
\]
Since $v$ was arbitrary, $\sigma$ is the identity and this establishes the injectivity of (\ref{InjectivePartitionMap}).

Next we observe that there exists a graph isomorphism
\begin{equation}\label{GraphIsomorphism}
\lambda:(\Lambda_1/H_1)\amalg (\Lambda_2/H_2)\amalg \dots \amalg (\Lambda_s/H_s)  \stackrel{\sim}{\rightarrow} \{\text{levels $1$ through $n$ of $T/G$}\}
\end{equation}
where each side is viewed as a disjoint union of $s$ rooted trees.  Here we define $\lambda$ to take the $H_i$-orbit of a point $v\in \Lambda_i$ to the $G$-orbit of $v$ in $T$.  This definition is unambiguous: the groups $G_i$ and $H_i=G_i/K_i$ have the same orbits in $\Lambda_i$ as elements of $K_i$ act trivially on $T_i$.  Therefore two points of $\Lambda_i$ lying in the same $H_i$-orbit also share a $G_i$-orbit, and hence they share the same $G$-orbit as well.  

To see that $\lambda$ is injective, consider $v\in \Lambda_i$ and $v'\in \Lambda_j$ with the same image under $\lambda$, thus $v$ and $v'$ are elements of the same $G$-orbit; that is $\sigma(v)=v'$ for some $\sigma\in G$.  Since $\sigma$ is a tree automorphism, this implies that the level 1 ancestors of $v$ and $v'$, which are $w_i$ and $w_j$, also satisfy $\sigma(w_i)=w_j$.  By construction of the representatives $w_1,w_2,\dots,w_s$ this can only happen when $i=j$.  Hence $\Lambda_i=\Lambda_j$.  Since $\sigma(w_i)=w_i$ we have $\sigma\in G_i$ and hence $v$ and $v'$ are $G_i$-conjugate, completing the proof that $\lambda$ is injective.  

To see that $\lambda$ is surjective, given a nonroot vertex $v\in T$, let $w$ be the level 1 ancestor of $v$.  Then $\sigma(w)=w_i$ for some $\sigma\in G$, and therefore $\sigma(v)\in \Lambda_i$.  Thus $\lambda$ takes the $G_i$-orbit of $\sigma(v)\in \Lambda_i$ to the $G$-orbit of $\sigma(v)$, which coincides with the $G$-orbit of $v$, completing the proof that $\lambda$ is surjective.

Recall that for each $i$, $K_i$ is the kernel of the action of $G_i$ on $\Lambda_i$ and $H_i=G_i/K_i$.  Therefore $\Lambda_i/G_i\simeq \Lambda_i/H_i$.  In particular the rooted trees $\Lambda_i/G_i$ and $\Lambda_i/H_i$ have the same number of incomplete vertices.  Since $H_i$ acts faithfully on $\Lambda_i$, it is isomorphic to an abelian subgroup of $\Aut(\Lambda_i)$, so by the induction hypothesis and the fact that $\Lambda_i$ has height $n-1$ we have
\begin{equation}\label{TiIncompleteBound}
|H_i| \leq q^{I(\Lambda_i/H_i)}.
\end{equation}

Define $\delta_{T/G}=0$ if the root of $T/G$ is complete, and $\delta_{T/G}=1$ if the root of $T/G$ is incomplete.  By the graph isomorphism (\ref{GraphIsomorphism}), the number of incomplete vertices of $T/G$ is the sum of the numbers of incomplete vertices of the $\Lambda_i/H_i$, plus one extra incomplete vertex if the root of $T/G$ is incomplete.  Using this fact as well as (\ref{FITTreeAutAbelian}), (\ref{R1BoundAbelian}), the injectivity of (\ref{InjectivePartitionMap}), and (\ref{TiIncompleteBound}) we have
\begin{equation*}
\begin{split}
\log_q|G| & =\log_q|G^*|+\log_q|H| \\
	& \leq \delta_{T/G}+\sum_{i=1}^{s}\log_q|H_i| \\
	& \leq \delta_{T/G}+\sum_{i=1}^{s}I(\Lambda_i/H_i) \\
	& = I(T/G),
\end{split}
\end{equation*}
completing the proof.
\end{proof}

\begin{proof}[Proof of Theorem \ref{AbelianSubgroupsAreSmallIntro}]
Let $T$ be an infinite $d$-ary rooted tree and let $G$ be an abelian subgroup of $\Aut(T)$.  We must show that $\dim(G)=0$.

For any infinite rooted $d$-ary tree $T$ and any $n\geq0$, let $I_n(T)$ denote the number of incomplete vertices of $T$ of level $n$.  Ab\'ert-Vir\'ag (\cite{MR2114819} Lemma 8.2) have shown that the set of incomplete vertices has density zero in the sense that
\begin{equation}\label{AbertViragLimit}
\lim_{n\to+\infty}\frac{I_n(T)}{d^n}=0.
\end{equation}

For each $n\geq0$ define $I^*_n(T)$ to be the number of incomplete vertices of the height $n$ subtree $T_n$ of $T$.  Equivalently, we have that 
\[
I^*_n(T)=I_0(T)+I_1(T)+I_2(T)+\dots+I_{n-1}(T)
\]
is the total number of incomplete vertices of $T$ among all levels $k$ with $0\leq k\leq n-1$.  The density result (\ref{AbertViragLimit}) implies that
\begin{equation}\label{AbertViragLimitStrengthened}
\lim_{n\to+\infty}\frac{I^*_n(T)}{d^n}=0.
\end{equation}
To see this, let $A_k=I_k(T)/d^k$ and let $A=\max_{k\geq0}A_k$, which is finite by (\ref{AbertViragLimit}).  Then for $n> m\geq0$ we have
\begin{equation*}
\begin{split}
\frac{I^*_n(T)}{d^n} & =\frac{A_{n-1}}{d}+\frac{A_{n-2}}{d^2}+\dots+\frac{A_{n-m}}{d^m} + \sum_{k=m+1}^{n}\frac{A_{n-k}}{d^k} \\
& \leq \frac{A_{n-1}}{d}+\frac{A_{n-2}}{d^2}+\dots+\frac{A_{n-m}}{d^m} + \sum_{k=m+1}^{\infty}\frac{A}{d^k}
\end{split}
\end{equation*}
Holding $m$ fixed and letting $n\to+\infty$, (\ref{AbertViragLimit}) implies that
\begin{equation*}
\begin{split}
\limsup_{n\to+\infty}\frac{I^*_n(T)}{d^n} & \leq \sum_{k=m+1}^{\infty}\frac{A}{d^k},
\end{split}
\end{equation*}
and as $m\geq0$ is arbitrary we conclude that $\limsup_{n\to+\infty}(I^*_n(T)/d^n)=0$, proving (\ref{AbertViragLimitStrengthened}).

For each $n\geq0$ recall that $r_n:\Aut(T)\to\Aut(T_n)$ denotes the surjective restriction map, where $T_n$ is the height $n$ subtree of $T$.  Let $G_n=r_n(G)$ be the image of $G$ under $r_n$; thus $G_n$ is an abelian subgroup of $\Aut(T_n)$.   Moreover, $T/G$ is a rooted tree whose height $n$ subtree $(T/G)_n$ is isomorphic to $T_n/G_n$.  In particular, $I_n^*(T/G)=I(T_n/G_n)$ as the incomplete vertices of $T/G$ among levels 0 through $n-1$ are in bijective correspondence with the set of all incomplete vertices of the height $n$ tree $T_n/G_n$.

Using Lemma \ref{MainAbelianSubgroupBound} and (\ref{AbertViragLimitStrengthened}) we have
\begin{equation*}
\begin{split}
\limsup_{n\rightarrow+\infty}\frac{1}{d^n}\log|G_n| & \leq \limsup_{n\rightarrow+\infty}\frac{1}{d^n}(\log q)I(T_n/G_n) \\
	& = \limsup_{n\rightarrow+\infty}\frac{1}{d^n}(\log q)I_n^*(T/G) \\
	& = 0.
\end{split}
\end{equation*}
We conclude from Lemma \ref{MinDimAbstractCalcProp} that $\udim(G)=0$.
\end{proof}


\section{The dimension of arboreal Galois groups}\label{AGRSection}

In this section we prove preliminary results about the Minkowski dimension of arboreal Galois groups.  Recall that an {\em arboreal pair} is a rational map $f:\PP^1\to\PP^1$ of degree $d\geq2$ defined over a field $K$ of characteristic zero or characteristic $>d$, together with a base point $\alpha\in\PP^1(K)$. 

\begin{prop}\label{MainArbHausdorffDimCalc}
Let $(f,\alpha)$ be an arboreal pair of degree $d$ defined over a field $K$.  The lower and upper Minkowski dimensions of $G_{f,\alpha}$ satisfy the identities 
\begin{equation*}
\begin{split}
\ldim(G_{f,\alpha})  & =\frac{d-1}{\log(d!)}\liminf_{n\to+\infty}\frac{1}{d^n}\log[K(f^{-n}(\alpha)):K] \\
\udim(G_{f,\alpha}) & =\frac{d-1}{\log(d!)}\limsup_{n\to+\infty}\frac{1}{d^n}\log[K(f^{-n}(\alpha)):K],
\end{split}
\end{equation*}
and 
\begin{equation*}
\dim(G_{f,\alpha})  =\frac{d-1}{\log(d!)}\lim_{n\to+\infty}\frac{1}{d^n}\log[K(f^{-n}(\alpha)):K] 
\end{equation*}
when this limit exists.
\end{prop}

\begin{proof}
Abbreviate $\rho=\rho_{f,\alpha}$, $G=G_{f,\alpha}$ and $T=T_{f,\alpha}$.  For each $n\geq0$, recall that $T_n$ is the finite subtree of $T$ including only the vertices in $L_0\cup L_1\cup L_2\cup \dots\cup L_n$ and the edges between these vertices.  Set $K_n=K(f^{-n}(\alpha))$ and $K_\infty=K(f^{-\infty}(\alpha))$.  We first show that 
\begin{equation}\label{HDDegreeIdentity}
[K_n:K]=|r_n(G)|
\end{equation}
where $r_n:\Aut(T)\to\Aut(T_n)$ denotes the surjective restriction map.  Both maps in the sequence
\begin{equation*}
\begin{CD}
\Gal(K_\infty/K)  @> \rho >> G @> r_n >> r_n(G)
\end{CD}
\end{equation*} 
are surjective, and the kernel of $r_n\circ\rho$ is the set of $K$-automorphisms $\sigma:K_\infty\to K_\infty$ fixing each point of $f^{-n}(\alpha)$.  Therefore $\ker(r_n\circ\rho)=\Gal(K_\infty/K_n)$.  It follows that
\[
\Gal(K_n/K)\simeq\Gal(K_\infty/K)/\Gal(K_\infty/K_n)\simeq r_n(G)
\]
from which (\ref{HDDegreeIdentity}) follows.  The desired identities for Minkowski dimension follow from (\ref{HDDegreeIdentity}) and Proposition \ref{MinDimAbstractCalcProp}.
\end{proof}

For the convenience of the reader we recall the Galois-theoretic diamond isomorphism theorem (\cite{lang:algebra} Thm. VI.1.12).

\begin{lem}\label{DITLemma}
Let $E/F$ be a field extension and let $F\subseteq F_1\subseteq E$ and $F\subseteq F_2\subseteq E$ be two intermediate extensions with $F_1/F$ Galois.  Then
\[
\Gal(F_1F_2/F_2)\simeq\Gal(F_1/(F_1\cap F_2))
\]
and in particular $[F_1F_2:F_2]=[F_1:(F_1\cap F_2)]\leq[F_1:F]$.
\end{lem}

The following proposition shows that the Minkowski dimensions of an arboreal Galois representation are invariant under finite extension of the base field.

\begin{prop}\label{BaseExtensionInvarianceProp}
Let $(f,\alpha)$ be an arboreal pair of degree $d$ defined over a field $K$.  Let $K\subseteq K'\subseteq\Kbar$ with $K'/K$ a finite extension.  Let $\rho_{f,\alpha}:\Gal(K(f^{-\infty}(\alpha))/K)\to\Aut(T_{f,\alpha})$ and $\rho'_{f,\alpha}:\Gal(K'(f^{-\infty}(\alpha))/K')\to\Aut(T_{f,\alpha})$ be the arboreal Galois representations defined over the base fields $K$ and $K'$, with images $G_{f,\alpha}$ and $G'_{f,\alpha}$ respectively.  Then 
\[
\ldim(G'_{f,\alpha})=\ldim(G_{f,\alpha})\hskip1cm\text{and}\hskip1cm\udim(G'_{f,\alpha})=\udim(G_{f,\alpha}).
\]
\end{prop}

\begin{proof}
The representation $\rho'_{f,\alpha}$ factors through $\rho_{f,\alpha}$ via the natural restriction embedding $\Gal(K'(f^{-\infty}(\alpha))/K')\hookrightarrow\Gal(K(f^{-\infty}(\alpha))/K)$.  It follows that $G'_{f,\alpha}$ is a subgroup of $G_{f,\alpha}$ of index $\leq [K':K]$.  The equalities of Minkowski dimension then follow from Proposition \ref{MinDimPropertiesProp}.
\end{proof}

Let $f:\PP^1\to \PP^1$ and $g:\PP^1\to \PP^1$ be rational maps defined over $K$ and let $L/K$ be a field extension.  We say that $f$ and $g$ are {\em $L$-conjugate} if there exists an automorphism $\varphi:\PP^1\to \PP^1$ defined over $L$ such that $\varphi\circ g\circ\varphi^{-1}=f$.  If such $\varphi:\PP^1\to \PP^1$ exists and $(f,\alpha)$ and $(g,\beta)$ are arboreal pairs satisfying $\varphi(\beta)=\alpha$, we say the arboreal pairs $(f,\alpha)$ and $(g,\beta)$ are {\em $L$-conjugate}.

\begin{prop}\label{ConjugacyInvarianceProp}
If the arboreal pairs $(f,\alpha)$ and $(g,\beta)$ are $\Kbar$-conjugate then $\ldim(G_{f,\alpha})=\ldim(G_{g,\beta})$ and $\udim(G_{f,\alpha})=\udim(G_{g,\beta})$.
\end{prop}

\begin{proof}
By Proposition \ref{BaseExtensionInvarianceProp} we may extend $K$ if necessary and assume without loss of generality that $\varphi$ is defined over $K$.  Since $\varphi\circ g\circ\varphi^{-1}=f$, we have $\varphi(g^{-n}(\beta))=f^{-n}(\alpha)$ for all $n\geq0$, which implies that $K(f^{-n}(\alpha))=K(g^{-n}(\beta))$ as $\varphi$ is defined over $K$.  The result then follows from Proposition \ref{MainArbHausdorffDimCalc}.
\end{proof}


\section{Dimension inequalities}\label{RationalMapsSect}

In this section we prove explicit versions of the dimension bounds summarized in Theorem \ref{NonMaximalityThmIntro}.

\begin{prop}\label{PostCritBasePointProp}
Let $(f,\alpha)$ be an arboreal pair of degree $d$ defined over a field $K$.  If $\alpha\in \PP^1(K)$ is an $f$-postcritical point and $m\geq1$ denotes the smallest positive integer such that $f^{-m}(\alpha)$ contains a critical point of $f$, then $\udim(G_{f,\alpha})\leq1-d^{-m}$.    
\end{prop}

\begin{proof}
The definition of $m$ implies that $|f^{-m}(\alpha)|\leq d^m-1$, and (\ref{NonCompleteUDimBound}) of the proof of Proposition \ref{NonCompleteProp} shows that $\udim(G_{f,\alpha})\leq |f^{-m}(\alpha)|/d^m=1-d^{-m}$.
\end{proof}

\begin{prop}\label{PeriodicBasePointProp}
Let $(f,\alpha)$ be an arboreal pair of degree $d$ defined over a field $K$.  If $\alpha\in \PP^1(K)$ is an $f$-periodic point and $m\geq1$ denotes the minimal $f$-period of $\alpha$, then $\udim(G_{f,\alpha})\leq1-d^{-m}$.    
\end{prop}

\begin{proof}
Let $C$ be the periodic $f$-cycle of size $m$ containing $\alpha$; thus $C\subseteq\PP^1(K)$.  Also $C\subseteq f^{-\infty}(\alpha)$ and the full backward orbit $f^{-\infty}(\alpha)$ contains no periodic points other than those in $C$, as every point in $f^{-\infty}(\alpha)$ contains $\alpha$ in its forward orbit.  If $C=f^{-\infty}(\alpha)$ then $\alpha$ is exceptional for $C$ and therefore $\dim(G_{f,\alpha})=0$ trivially, so we may assume that $f^{-\infty}(\alpha)$ contains some strictly preperiodic points.  

We may assume without loss of generality that $|f^{-m}(\alpha)|=d^m$, since if $|f^{-m}(\alpha)|<d^m$ then we already know from Proposition \ref{PostCritBasePointProp} that $\udim(G_{f,\alpha})\leq 1-d^{-m}$. Enumerate 
\[
f^{-m}(\alpha)=\{\beta_1,\beta_2,\dots,\beta_{d^m}\}.
\]
Since $\alpha$ is $m$-periodic we have $\alpha\in f^{-m}(\alpha)$, and therefore we may index the $\beta_i$ so that $\beta_{d^m}=\alpha$.  By Proposition \ref{BaseExtensionInvarianceProp} we may extend the base field if needed and assume that all of the $\beta_1,\beta_2,\dots,\beta_{d^m}$ are $K$-rational, and hence every point indexing levels $0$ through $m$ of the preimage tree $T_{f,\alpha}$ is $K$-rational.  

Fix $n\geq m$, and set $k=n-m$.  Thus we have a decomposition
\begin{equation}\label{PreimageBranches}
f^{-n}(\alpha)=f^{-k}(\beta_1)\cup f^{-k}(\beta_2)\cup \dots \cup f^{-k}(\beta_{d^m})
\end{equation}
as each $\beta\in f^{-n}(\alpha)$ satisfies $f^k(\beta)=\beta_i$ for some $1\leq i\leq d^m$.  We will show that 
\begin{equation}\label{FewerGenerators}
K(f^{-n}(\alpha)) = K(f^{-k}(\beta_1)\cup f^{-k}(\beta_2)\cup \dots \cup f^{-k}(\beta_{d^m-1})).
\end{equation}
The significance of (\ref{FewerGenerators}) is that this field identity is still true, even though the generating set on the right-hand-side omits the final branch $f^{-k}(\beta_{d^m})=f^{-k}(\alpha)$ of the preimage tree, which one might expect a priori to be necessary in view of the decomposition (\ref{PreimageBranches}).  This can be viewed as a consequence of the periodicity occurring in the tree $T_{f,\alpha}$, which is caused by the root $\alpha$ appearing in the 0-th level and again in the $m$-th level as $\beta_{d^m}=\alpha$.

To prove (\ref{FewerGenerators}) it suffices to prove that 
\begin{equation}\label{FewerGenerators2}
f^{-k}(\beta_{d^m})\subseteq \PP^1(L)
\end{equation}
for
\[
L=K(f^{-k}(\beta_1)\cup f^{-k}(\beta_2)\cup \dots \cup f^{-k}(\beta_{d^m-1})),
\]
after which (\ref{FewerGenerators}) follows from this and (\ref{PreimageBranches}), as it shows that the points of $f^{-k}(\beta_{d^m})$ are already defined over the field generated by the other branches $f^{-k}(\beta_i)$ for $i=1,2,\dots,d^m-1$.  Let $\beta\in f^{-k}(\beta_{d^m})=f^{-k}(\alpha)$, thus $f^k(\beta)=\alpha$.  If $\beta$ is $f$-periodic, then $\beta\in C$ and hence $\beta$ is already $K$-rational and hence in $L$.  So we may assume that $\beta$ is not $f$-periodic. 

\medskip

\noindent
{\bf Claim:} For some $1\leq i\leq d^m-1$, $\beta$ is an element of the forward $f$-orbit of some $\beta'$ in $f^{-k}(\beta_i)$.

\medskip

If we can show the Claim, then it follows at once that $\beta\in \PP^1(L)$, as $f$ is defined over $K$, which completes the proof of (\ref{FewerGenerators2}) and hence (\ref{FewerGenerators}).

To prove the Claim, define the {\em preperiod} of $\beta$ to be the smallest integer $\ell\geq1$ for which $f^\ell(\beta)\in C$; this preperiod is positive as we have an assumption that $\beta$ itself is not periodic.  Also $\ell\leq k$ since $f^k(\beta)=\alpha\in C$.  Let $\gamma=f^{\ell-1}(\beta)$ be the final non-periodic point in the forward orbit of $\beta$.  Then 
\[
f(\gamma)\in C\subseteq\{\alpha\}\cup f^{-1}(\alpha)\cup \dots\cup f^{-(m-1)}(\alpha),
\]
which implies that
\[
\gamma\in f^{-1}(\alpha)\cup f^{-2}(\alpha)\cup \dots\cup f^{-m}(\alpha).
\]

Since $\gamma$ occurs in one of the sets 
\[
f^{-1}(\alpha),f^{-2}(\alpha),\dots,f^{-m}(\alpha)=\{\beta_1,\beta_2,\dots,\beta_{d^m}\},
\]
it follows that $\gamma$ is an element of the forward $f$-orbit $\{f^j(\beta_i)\mid j\geq0\}$ of one of the $\beta_i$.  But $\gamma$ is not periodic, so $\gamma$ is not in the forward $f$-orbit of the periodic point $\beta_d^m=\alpha$.  So $\gamma$ is an element of the forward $f$-orbit of $\beta_i$ for some $i=1,2,\dots, d^m-1$.  Since $\gamma=f^{\ell-1}(\beta)$ and $\ell\leq k$, it follows that $\beta$ itself is in the forward $f$-orbit of some $\beta'$ in $f^{-k}(\beta_i)$, completing the proof of the Claim and hence of (\ref{FewerGenerators}).  

Finally, with (\ref{FewerGenerators}) proved, we can complete the proof of the theorem.  For each $1\leq i\leq d^m-1$ we have 
\begin{equation}\label{LevelkTreeBound}
[K(f^{-k}(\beta_i)):K]=|\Gal(K(f^{-k}(\beta_i))/K)|\leq d!^{(d^k-1)/(d-1)}.
\end{equation}
Here we have used the fact that $\beta_i$ is $K$-rational and hence the Galois group $\Gal(K(f^{-k}(\beta_i))/K)$ is isomorphic to a subgroup of $\Aut(T_k)$, where $T=T_{f,\beta_i}$, and hence (\ref{AutGroupCountBound}) implies the stated inequality in (\ref{LevelkTreeBound}).  

Using (\ref{FewerGenerators}) and the well-known submultiplicative property of the degrees of field extensions under compositum we have 
\begin{equation*}
\begin{split}
\frac{1}{d^n}\log[K(f^{-n}(\alpha)):K] & \leq \frac{1}{d^n}\log \prod_{i=1}^{d^m-1}[K(f^{-k}(\beta_i)):K] \\
	& \leq \frac{d^m-1}{d^n}\log(d!^{(d^k-1)/(d-1)}).
\end{split}
\end{equation*}
Recalling that $k=n-m$, we obtain
\begin{equation*}
\limsup_{n\to+\infty}\frac{1}{d^n}\log[K(f^{-n}(\alpha)):K] \leq (1-d^{-m})\frac{\log d!}{(d-1)}
\end{equation*}
and hence $\udim(G_{f,\alpha})\leq 1-d^{-m}$ using Proposition \ref{MainArbHausdorffDimCalc}.
\end{proof}

The following proposition gives a bound on the Minkowski dimension of an arboreal Galois representation associated to an iterate of a rational map.  

\begin{prop}\label{IterateDimProp}
Let $(f,\alpha)$ be an arboreal pair of degree $d$ defined over a field $K$.  
\begin{itemize}
\item[{\bf (a)}]  For each $m\geq2$ it holds that 
\[
\udim(G_{f^m,\alpha}) \leq (C_{d}/C_{d^m})\udim(G_{f,\alpha}),
\]
where $C_{D}=\frac{\log (D!)}{D-1}$.  In particular, $\udim(G_{f^m,\alpha}) \leq C_{d}/C_{d^m} <1$.
\item[{\bf (b)}] If $\dim(G_{f,\alpha})$ exists, then $\dim(G_{f^m,\alpha}) = (C_{d}/C_{d^m})\dim(G_{f,\alpha})$.
\end{itemize}
\end{prop}

\begin{proof}
Let $m\geq2$.  The rational map $g=f^m$ has degree $D=d^m$, so using Proposition \ref{MainArbHausdorffDimCalc} we have
\begin{equation*}
\begin{split}
C_{D} \, \udim(G_{g,\alpha}) &  =  \limsup_{n\to+\infty}\frac{\log[K(g^{-n}(\alpha)):K]}{D^n} \\
	&  = \limsup_{n\to+\infty}\frac{\log[K(f^{-mn}(\alpha)):K]}{d^{mn}} \\
	&  \leq \limsup_{n\to+\infty}\frac{\log[K(f^{-n}(\alpha)):K]}{d^{n}} \\
	&  = C_d \, \udim(G_{f,\alpha}),
\end{split}
\end{equation*}
which proves the first statement of part {\bf (a)}.  We observe that $C_D=\frac{\log(D!)}{D-1}$ is increasing for $D\geq2$ and therefore $C_{d}/C_{d^m}<1$ whenever $m\geq2$.  If $\dim(G_{f,\alpha})$ exists, then each $\limsup$ in the preceding calculation is a limit and equality holds in the inequality, from which part {\bf (b)} follows.
\end{proof}

A consequence of Proposition \ref{IterateDimProp} is that if $\dim(G_{f,\alpha})=0$, then $\dim(G_{f^m,\alpha})=0$ for all $m\geq1$.  The following lemma is a sort of converse to this statement, as it allows one to derive $\dim(G_{f,\alpha})=0$ provided that $\dim(G_{f^m,\beta})=0$ for a suitable finite collection of base points $\beta$ in the backward $f$-orbit of $\alpha$.

\begin{lem}\label{IterateDimPropAlternate}
Let $(f,\alpha)$ be an arboreal pair of degree $d$ defined over a field $K$.  Let $m\geq2$ and define
\[
\Ucal_m=\{\alpha\}\cup f^{-1}(\alpha) \cup f^{-2}(\alpha)\cup\dots\cup f^{-(m-1)}(\alpha).
\]
If $\Ucal_m\subseteq\PP^1(K)$ and $\dim(G_{f^m,\beta})=0$ for all $\beta\in \Ucal_m$, then $\dim(G_{f,\alpha})=0$.
\end{lem}

\begin{proof}
Define $a_n=\frac{1}{d^n}\log[K(f^{-n}(\alpha)):K]$.  Partitioning the sequence $\{a_n\}_{n=0}^{\infty}$ into $m$ subsequences $\{a_{mk+\ell}\}_{k=0}^{\infty}$ over $\ell=0,1,2,\dots,m-1$, the pigeonhole principle implies that the largest limit point of $\{a_n\}_{n=0}^{\infty}$ must be equal to the largest real number occurring as a limit point of one of these $m$ subsequences.  Therefore from Proposition \ref{MainArbHausdorffDimCalc} we have
\begin{equation}\label{PartitionMthPreimage}
\begin{split}
\frac{\log (d!)}{d-1} \, \udim(G_{f,\alpha}) &  =  \limsup_{n\to+\infty}a_n =\max_{0\leq \ell\leq m-1}\left(\limsup_{k\to+\infty}a_{mk+\ell}\right).
\end{split}
\end{equation}

Fix $\ell$ in the range $0\leq \ell\leq m-1$, and enumerate 
\[
f^{-\ell}(\alpha)=\{\beta_1,\beta_2,\dots,\beta_{d^\ell}\}\subseteq\PP^1(K),
\]
repeating according to multiplicity.  Set $g=f^m$, thus $\deg(g)=d^m$.  We have
\begin{equation*}
f^{-(mk+\ell)}(\alpha)=g^{-k}(\beta_1)\cup g^{-k}(\beta_2)\cup \dots\cup g^{-k}(\beta_{d^\ell})
\end{equation*}
and therefore 
\[
K(f^{-(mk+\ell)}(\alpha))=\prod_{i=1}^{d^\ell} K(g^{-k}(\beta_i)).
\]
Using the well-known submultiplicative property of the degrees of field extensions under compositum we have 
\begin{equation}\label{ellthPartitionMthPreimage}
\begin{split}
a_{mk+\ell} & = \frac{1}{d^{mk+\ell}}\log[K(f^{-(mk+\ell)}(\alpha)):K] \\
	& \leq \frac{1}{d^{mk+\ell}}\log\prod_{i=1}^{d^\ell}  [K(g^{-k}(\beta_i)):K] \\
	& = \frac{1}{d^\ell}\sum_{i=1}^{d^\ell}  \frac{1}{(d^m)^k}\log[K(g^{-k}(\beta_i)):K] \\
	& \to 0 \text{ as } k\to+\infty,
\end{split}
\end{equation}
by the hypothesis that $\udim(G_{f^m,\beta_i})=0$ for all $\beta_i\in f^{-\ell}(\alpha)\subseteq\Ucal_m$. We conclude from (\ref{ellthPartitionMthPreimage}) and (\ref{PartitionMthPreimage}) that $\udim(G_{f,\alpha})=0$, completing the proof.
\end{proof}


\section{The dimension of profinite iterated monodromy groups}\label{IMGSect}

In this section we prove Theorems \ref{IMGDimExistsEqualThmIntro} and \ref{IMGDichotomyThmIntro}. Throughout this section, $f:\PP^1\to\PP^1$ is a rational map of degree $d\geq2$ defined over a field $K$ of characteristic zero or $>d$, and $t$ is an indeterminate. 

The following Lemma reflects the self-similarity of the group $G_{f}^\arith$ and contains the main idea behind the proofs of Theorems \ref{IMGDimExistsEqualThmIntro} and \ref{IMGDichotomyThmIntro}.

\begin{lem}\label{IMGLimsupLemma}
For each $n\geq0$ define 
\[
a_n=\frac{1}{d^n}\log[K(f^{-n}(t)):K(t)].
\]
Then for each $m\geq1$ it holds that 
\begin{equation}\label{IMGLimsupLemmaBound}
\limsup_{n\to+\infty} (a_n) \leq \frac{a_m}{1-\frac{1}{d^m}}.  
\end{equation}
\end{lem}

\begin{proof}
Fix $m\geq1$ and to ease notation set
\[
M:=[K(f^{-m}(t)):K(t)].
\]
Thus $a_m=\frac{1}{d^m}\log M$.  We claim that if $L/K$ is any field extension and $\beta\in \PP^1(L)$, then 
\begin{equation}\label{AnyExtensionBoundNew}
[L(f^{-m}(\beta)):L]\leq M.
\end{equation}
To prove this, we first apply Lemma \ref{DITLemma} with $F = K(t)$, $F_1  = K(f^{-m}(t))$, and $F_2 = L(t)$.  Then $F_1F_2=L(f^{-m}(t))$ and the bound $[F_1F_2:F_2]\leq[F_1:F]$ of Lemma \ref{DITLemma} implies that \begin{equation}\label{LBoundGenericNew}
[L(f^{-m}(t)):L(t)]\leq[K(f^{-m}(t)):K(t)]=M.
\end{equation}
Since the degree of an extension of function fields is non-increasing under specialization, we have 
\begin{equation}\label{SpecializationNonIncreasingNew}
[L(f^{-m}(\beta)):L] \leq [L(f^{-m}(t)):L(t)],
\end{equation}
and together (\ref{LBoundGenericNew}) and (\ref{SpecializationNonIncreasingNew}) complete the proof of (\ref{AnyExtensionBoundNew}).

Next, we use (\ref{AnyExtensionBoundNew}) for all $\beta$ in the backward $f$-orbit of $t$.  For each $n\geq0$ set $K_n=K(f^{-n}(t))$.  Enumerate $f^{-n}(t)=\{\beta_1,\beta_2,\dots,\beta_{d^n}\}$.  Thus
\[
K_{n+m}=K_n(f^{-m}(\beta_1))K_n(f^{-m}(\beta_2))\dots K_n(f^{-m}(\beta_{d^n})).
\]
Using (\ref{AnyExtensionBoundNew}) with $L=K_n$ and the well-known submultiplicative property of the degrees of field extensions under compositum we have 
\begin{equation}\label{KnIterativeBoundNew}
\begin{split}
[K_{n+m}:K_n] & \leq \prod_{i=1}^{d^n}[K_{n}(f^{-m}(\beta_i)):K_n] \leq M^{d^n}.
\end{split}
\end{equation}
Using (\ref{KnIterativeBoundNew}) and the fact that $a_m=\frac{1}{d^m}\log M$ we obtain
\begin{equation}\label{anIterativeBoundNew}
\begin{split}
a_{n+m} & =\frac{1}{d^{n+m}}\log[K_{n+m}:K] \\
	& =\frac{1}{d^{n+m}}\log[K_{n+m}:K_n]+\frac{1}{d^{n+m}}\log[K_{n}:K] \\
	& \leq \frac{1}{d^m}\log M+\frac{a_n}{d^m} \\
	& = a_m+\frac{a_n}{d^m}
\end{split}
\end{equation}
Taking a limit supremum in (\ref{anIterativeBoundNew}) with respect to $n$ (with $m$ still fixed) we obtain
\begin{equation*}
\begin{split}
\limsup_{n\to+\infty}(a_{n})  & = \limsup_{n\to+\infty}(a_{n+m}) \\
	& \leq a_m+\frac{1}{d^m}\limsup_{n\to+\infty}(a_n),
\end{split}
\end{equation*}
which is equivalent to (\ref{IMGLimsupLemmaBound}) and completes the proof.
\end{proof}

\begin{proof}[Proof of Theorem \ref{IMGDimExistsEqualThmIntro}]
For each intermediate field $K\subseteq L\subseteq \Ksep$ and each $n\geq0$, define 
\[
a_n(L)=\frac{1}{d^n}\log[L(f^{-n}(t)):L(t)].
\]
For each fixed $m\geq1$ we apply Lemma \ref{IMGLimsupLemma} (with the field $L$ in place of $K$) to obtain 
\begin{equation}
\limsup_{n\to+\infty} a_n(L) \leq \frac{a_m(L)}{1-\frac{1}{d^m}}.  
\end{equation}
Since the left hand side does not depend on $m$, we may take a limit infimum with respect to $m$ on the right hand side and obtain  
\[
\limsup_{n\to+\infty}a_n(L)\leq \liminf_{m\to+\infty}\frac{a_m(L)}{1-\frac{1}{d^m}}=\liminf_{m\to+\infty}a_m(L)
\]
from which it follows that $\lim_{n\to+\infty}a_n(L)$ exists.  The two extremal cases $L=\Ksep$ and $L=K$ imply by Proposition \ref{MainArbHausdorffDimCalc} that
\begin{equation*}
\begin{split}
\dim(G_{f}^\geom) & = \frac{d-1}{\log(d!)}\lim_{n\to+\infty}a_n(\Ksep) \\
\dim(G_{f}^\arith) & = \frac{d-1}{\log(d!)}\lim_{n\to+\infty}a_n(K)
\end{split}
\end{equation*}
both exist.  

It remains to prove that $\dim(G_f^\geom)=\dim(G_f^\arith)$.  We know from the inclusion (\ref{GeomArithInclusion}) that $\dim(G_f^\geom)\leq\dim(G_f^\arith)$, so it suffices to show that $\dim(G_f^\arith)\leq\dim(G_f^\geom)$, which is the same as showing that
\begin{equation}\label{TheGoal}
\lim_{n\to+\infty}a_n(K)\leq \lim_{n\to+\infty}a_n(\Ksep).
\end{equation}

We observe that if $L/K$ is any finite extension then 
\begin{equation}\label{DoesNotDependOnL}
\lim_{n\to+\infty}a_n(K) = \lim_{n\to+\infty}a_n(L),
\end{equation}
because $L(t)/K(t)$ is a finite extension, and Minkowski dimension is invariant under finite extension of the base field by Proposition \ref{BaseExtensionInvarianceProp}.

Fix $m\geq1$.  Note that $K(f^{-m}(t))/K$ is a separable (hence Galois) extension by our standing assumption on the characteristic of $K$.  Let $L=\Ksep\cap K(f^{-m}(t))$ be the field of constants in $K(f^{-m}(t))$.  Thus $L/K$ is a finite extension and
\begin{equation}\label{ApplicationOfDIT}
[\Ksep(f^{-m}(t)):\Ksep(t)]=[L(f^{-m}(t)):L(t)],
\end{equation}
which can be seen via the diamond isomorphism theorem, stated as Lemma \ref{DITLemma}, with $F=K(t)$, $F_1=K(f^{-m}(t))=L(f^{-m}(t))$, and $F_2=\Ksep(t)$.  

Next, note that (\ref{ApplicationOfDIT}) is equivalent to the identity $a_m(L)=a_m(\Ksep)$.  Applying Lemma \ref{IMGLimsupLemma} (with the field $L$ in place of $K$) we have
\begin{equation}\label{UsingIMGLimsupLemmaBound}
\lim_{n\to+\infty} a_n(L) \leq \frac{a_m(L)}{1-\frac{1}{d^m}} = \frac{a_m(\Ksep)}{1-\frac{1}{d^m}}.
\end{equation}
Combining (\ref{UsingIMGLimsupLemmaBound}) with (\ref{DoesNotDependOnL}) we have
\begin{equation}\label{PreLimitInm}
\lim_{n\to+\infty} a_n(K) \leq \frac{a_m(\Ksep)}{1-\frac{1}{d^m}}.
\end{equation}
As $m\geq1$ is arbitrary and the left-hand-side of (\ref{PreLimitInm}) does not depend on $m$, we may take a limit on the right-hand-side to obtain the desired inequality (\ref{TheGoal}) and completing the proof.
\end{proof}

With the existence of $\dim(G_{f}^\arith)$ established, we prove the bound (\ref{MDIMGUpperBound}) stated in the introduction.  If $\alpha$ is not postcritical, then this can be seen easily from the inclusion (\ref{SpecializationIntro}).  But generally, for each $n\geq0$, as the degree of an extension of function fields is non-increasing under specialization, it holds that
\begin{equation*}
[K(f^{-n}(\alpha)):K] \leq [K(f^{-n}(t)):K(t)],
\end{equation*}
and $\udim(G_{f,\alpha})\leq\dim(G_{f}^\arith)$ then follows from this and Proposition \ref{MainArbHausdorffDimCalc}.



\begin{prop}\label{NonLargeArithIMGProp}
If $G_{f}^\arith\neq\Aut(T_{f,t})$, then $\dim(G_{f}^\arith)<1$.
\end{prop}

\begin{proof}
Let $T=T_{f,t}$, a complete infinite $d$-ary rooted tree, and for each $n\geq0$, let $T_n$ denote the finite $d$-ary rooted tree of height $n$ obtained by selecting only the vertices of $T$ of levels $k=0,1,2,\dots,n$, and the edges between them.  The hypothesis that $G_{f}^\arith\neq\Aut(T)$ means that there exists $m\geq1$ such that the image of 
\[
\rho_{f}:\Gal(K(f^{-\infty}(t))/K(t))\to\Aut(T)
\]
restricts to a proper subgroup of the full automorphism group of the finite tree $T_m$.  Thus 
\begin{equation*}
\begin{split}
[K(f^{-m}(t)):K(t)] & = |\Gal(K(f^{-m}(t))/K(t))|\\
	& <|\Aut(T_m)| \\
	& =d!^{(d^m-1)/(d-1)}.
\end{split}
\end{equation*}
By Lemma \ref{IMGLimsupLemma} we have
\begin{equation*}
\begin{split}
\lim_{n\to+\infty} \frac{1}{d^n}\log[K(f^{-n}(t)):K(t)] & \leq \frac{1}{(1-\frac{1}{d^m})}\frac{1}{d^m}\log[K(f^{-m}(t)):K(t)] \\
	& < \frac{1}{(1-\frac{1}{d^m})}\frac{1}{d^m}\log(d!^{(d^m-1)/(d-1)}) \\
	& = \frac{\log d!}{d-1},
\end{split}
\end{equation*}
and $\dim(G_f^\arith)<1$ follows from Proposition \ref{MainArbHausdorffDimCalc}.
\end{proof}

\begin{proof}[Proof of Theorem \ref{IMGDichotomyThmIntro}]
We must show that one and only one of these two conditions holds:
\begin{itemize}
\item[{\bf (i)}]   $G_{f}^\geom=G_{f}^\arith=\Aut(T_{f,t})$; or
\item[{\bf (ii)}]   $\dim(G_{f}^\geom)=\dim(G_{f}^\arith)<1$.
\end{itemize}

Conditions {\bf (i)} and {\bf (ii)} are mutually exclusive since $\dim(T_{f,t})=1$.  Assuming that {\bf (ii)} is false, we will prove that {\bf (i)} holds.  Because $\dim(G_{f}^\geom)=\dim(G_{f}^\arith)$, the assumption that {\bf (ii)} is false means that $\udim(G_{f}^\arith)=1$.  From Proposition \ref{NonLargeArithIMGProp} we deduce that $G_{f}^\arith=\Aut(T_{f,t})$.  That $G_{f}^\geom=\Aut(T_{f,t})$ follows as well, since $G_f^\geom$ is the special case of $G_f^\arith$ viewing $\Ksep$ as the base field.
\end{proof}



\section{Dimension maximality for quadratic rational maps}\label{MaxConjSect}

We have seen in Proposition \ref{PostCritBasePointProp} that if $(f,\alpha)$ is an arboreal pair, then dimension maximality $\dim(G_{f,\alpha})=1$ cannot occur in the case of postcritical base point $\alpha$.  So for the remainder of this section we assume that $\alpha$ is not $f$-postcritical.  In this case, $\dim(\Aut(T_{f,\alpha}))=1$ and 
\begin{equation}\label{FIImpliesMaximal}
[\Aut(T_{f,\alpha}):G_{f,\alpha}]<+\infty \hskip1cm\Rightarrow\hskip1cm \dim(G_{f,\alpha})=1
\end{equation} 
by Proposition \ref{MinDimPropertiesProp}.  We do not know if the converse of (\ref{FIImpliesMaximal}) is true in general, but in this section we show that if Jones' finite index conjecture (\cite{MR3220023} Conjecture 3.11) is true, then the converse of (\ref{FIImpliesMaximal}) holds for quadratic rational maps over number fields.

Let $(f,\alpha)$ be an arboreal pair of degree $d=2$ defined over $K$, such that $\alpha$ is not $f$-postcritical, and consider the following possibilities:
\begin{itemize}
\item[{\bf (i)}] The base point $\alpha$ is $f$-periodic.
\item[{\bf (ii)}] The rational map $f$ is postcritically finite.
\item[{\bf (iii)}] A critical orbit collision of the form $f^{r+1}(\gamma_1)=f^{r+1}(\gamma_2)$ holds for some $r\geq1$, where $\gamma_1,\gamma_2\in\PP^1(\Kbar)$ are the critical points of $f$.
\item[{\bf (iv)}] There exists a nontrivial $\Kbar$-automorphism $\varphi:\PP^1\to\PP^1$ such that $f\circ\varphi=\varphi\circ f$ and $\varphi(\alpha)=\alpha$.
\end{itemize}
Jones' finite index conjecture (\cite{MR3220023} Conjecture 3.11) states (still assuming $d=2$) that if $K$ is a number field and if the pair $(f,\alpha)$ satisfies none of the four conditions listed above, then $G_{f,\alpha}$ must have finite index as a subgroup of $\Aut(T_{f,\alpha})$.

For any field $K$ (but still assuming that $\ch(K)$ is zero or $>2)$, it was observed by Jones (\cite{MR3220023} $\S$3) that $[\Aut(T_{f,\alpha}):G_{f,\alpha}]=+\infty$ under any of the four conditions {\bf (i)}, {\bf (ii)}, {\bf (iii)}, {\bf (iv)}.  But in each of these cases, infinite index can actually be strengthened to dimension non-maximality $\udim(G_{f,\alpha})<1$, as follows.  Conditions {\bf (i)} and {\bf (ii)} imply that $\udim(G_{f,\alpha})<1$ by Proposition \ref{PeriodicBasePointProp} and Corollary \ref{PCFCorIntro}, respectively.  Under condition {\bf (iii)}, Pink has shown (\cite{pink2} Thm. 4.8.1(b)) that $G_{f}^\geom\neq\Aut(T_{f,t})$ and hence $\udim(G_{f,\alpha})\leq\udim(G_{f}^\arith)<1$ by (\ref{MDIMGUpperBound}) and Theorem \ref{IMGDichotomyThmIntro}.  And under condition {\bf (iv)}, as explained in \cite{MR3220023} $\S$3 we may extend the base field $K$ and conjugate $f$ with a $K$-automorphism of $\PP^1$, so without loss of generality we can reduce to the case $f(x)=\frac{r(x^2+1)}{x}$ for some nonzero $r\in K$, and $\alpha=0$.  Because $f\circ\psi=\psi\circ f$ for $\psi(x)=-x$, it follows that $G_{f,0}\subseteq C(\psi)\subseteq\Aut(T_{f,0})$ where $C(\psi)$ is the centralizer of $\psi$ viewed as an automorphism of $T_{f,0}$.  It follows that $\udim(G_{f,0})\leq \udim (C(\varphi))=\frac{1}{2}$ (see \cite{MR3177912} Cor. 4.2).

\begin{thm}
Assume that Jones' finite index conjecture (\cite{MR3220023} Conjecture 3.11) is true.  Let $(f,\alpha)$ be an arboreal pair of degree $d=2$ defined over a number field $K$ and assume that $\alpha$ is not $f$-postcritical.  Then 
\begin{equation}\label{FIIIFMaximal}
[\Aut(T_{f,\alpha}):G_{f,\alpha}]<+\infty \hskip1cm\text{ if and only if } \hskip1cm\dim(G_{f,\alpha})=1.
\end{equation} 
\end{thm}

\begin{proof}
The forward implication of (\ref{FIIIFMaximal}) follows from Proposition \ref{MinDimPropertiesProp}.  For the converse, assume that $\dim(G_{f,\alpha})=1$.  Then the pair $(f,\alpha)$ does not satisfy any of the conditions {\bf (i)}, {\bf (ii)}, {\bf (iii)}, {\bf (iv)}, as we have seen that each of these force dimension non-maximality $\udim(G_{f,\alpha})<1$.  By Jones' conjecture, we conclude that $[\Aut(T_{f,\alpha}):G_{f,\alpha}]<+\infty$.
\end{proof}


\section{Dimension minimality}\label{MinConjSect}

We begin this section by proving that power, Chevyshev, and Latt\`es maps satisfy dimension minimality $\dim(G_{f,\alpha})=0$.  

A rational map $f:\PP^1\to\PP^1$ of degree $d\geq2$ defined over a field $K$ is a {\em power map} if $f$ is $\Kbar$-conjugate to $x^{\pm d}$.  

\begin{prop}\label{PowerMapsAreSmall}
Let $(f,\alpha)$ be an arboreal pair of degree $d\geq2$ over a field $K$.  If $f$ is a power map, then $\dim(G_{f,\alpha})=0$.
\end{prop}

\begin{proof}
We are given that $f$ is $\Kbar$-conjugate to $x^{\pm d}$.  By Proposition \ref{BaseExtensionInvarianceProp} we may extend the base field if needed and assume that $f$ is $K$-conjugate to $x^{\pm d}$.  Then invoking Proposition \ref{ConjugacyInvarianceProp}, we may simply assume without loss of generality that $f(x)=x^{\pm d}$.

First consider the case $f(x)=x^{d}$.  Fix $n\geq0$ and let $K_n=K(f^{-n}(\alpha))$.  Fix $\beta\in f^{-n}(\alpha)\subseteq K_n$, thus $\beta^{d^n}=\alpha$ and more generally $f^{-n}(\alpha)$ is the set of roots in $\Kbar$ of $x^{d^n}-\alpha$.  By a well-known calculation in Kummer theory, $K_n$ contains a primitive $d^n$-th root of unity $\zeta$ and
\[
f^{-n}(\alpha)=\{\beta,\zeta\beta,\zeta^2\beta,\dots,\zeta^{d^n-1}\beta\},
\]
so $K_n=K(\zeta,\beta)$.  In particular,
\[
[K_n:K]\leq [K(\zeta):K][K(\beta):K]\leq \varphi(d^n)d^n\leq d^{2n}.
\]
The desired limit $\udim(G_{f,\alpha})=0$ now follows from this and Proposition \ref{MainArbHausdorffDimCalc}.

When $f(x)=x^{-d}$, we have $f^2(x)=x^{d^2}$.  Extending the base field $K$ if necessary, the hypotheses of Lemma \ref{IterateDimPropAlternate} are satisfied with $m=2$ because we have already proved the vanishing of $\udim(G_{f^2,\beta})$ for all $\beta\in\PP^1(K)$.  Therefore $\udim(G_{f,\alpha})=0$ by Lemma \ref{IterateDimPropAlternate}, completing the proof.
\end{proof}

A rational map $f:\PP^1\to\PP^1$ of degree $d\geq2$ defined over a field $K$ is a {\em Chebyshev map} if $f$ is $\Kbar$-conjugate to $\pm T_{d}(x)$, where $T_d(x)$ denotes the $d$-th Chebyshev polynomial.  Recall (\cite{MR2316407} $\S$6.2) that the Chebyshev polynomial $T_d(x)\in\ZZ[x]$ is characterized by the identity $T_d(x+x^{-1}) = x^d+x^{-d}$, which can be interpreted as a semi-conjugacy
\begin{equation}\label{ChebCharacterization2}
\begin{CD}
\PP^1 @> x^d >> \PP^1 \\ 
@V \pi VV                                    @VV \pi V \\ 
\PP^1  @> T_d >>  \PP^1
\end{CD}
\hskip2cm  T_d\circ\pi=\pi\circ x^d 
\end{equation}
from the power map $g(x)=x^d$ to $T_d(x)$, where $\pi:\PP^1\to\PP^1$ is the rational map $\pi(x)=x+x^{-1}$.

\begin{prop}\label{ChebyshevMapsAreSmall}
Let $(f,\alpha)$ be an arboreal pair of degree $d\geq2$ over a field $K$.  If $f$ is a Chebyshev map, then $\dim(G_{f,\alpha})=0$.
\end{prop}

\begin{proof}
As in the case of power maps, applications of Propositions \ref{BaseExtensionInvarianceProp} and \ref{ConjugacyInvarianceProp} allow us to reduce to the case that $f(x)=\pm T_{d}(x)$.  

\medskip
\noindent
\underline{Case 1: $f(x)=T_{d}(x)$.}  Extending the base field if necessary using Proposition \ref{BaseExtensionInvarianceProp}, we may assume that $\pi^{-1}(\alpha)=\{\beta_1,\beta_2\}\subseteq \PP^1(K)$.   We first show that
\begin{equation}\label{SemiConjFieldInclusion}
K(f^{-n}(\alpha))\subseteq K(g^{-n}(\beta_1)\cup g^{-n}(\beta_2))
\end{equation}
for each $n\geq0$.  Given $\gamma\in f^{-n}(\alpha)$, select $\delta\in\pi^{-1}(\gamma)$.  The semi-conjugacy (\ref{ChebCharacterization2}) implies that $f^n\circ\pi=\pi\circ g^n$, and therefore
\[
\alpha=f^n(\gamma) = f^n(\pi(\delta)) = \pi(g^n(\delta)).
\]
It follows that $g^{n}(\delta)=\beta_i$ for some $i=1,2$; that is $\delta\in g^{-n}(\beta_i)$.  We have shown that $\gamma$ is the $\pi$-image of an element of $g^{-n}(\beta_1)\cup g^{-n}(\beta_2)$, and as $\pi$ is defined over $K$, we obtain
\[
\gamma\in K(g^{-n}(\beta_1)\cup g^{-n}(\beta_2))
\]
from which (\ref{SemiConjFieldInclusion}) follows.

Using the submultiplicative property of the degrees of field extensions under compositum we have 
\begin{equation}\label{SemiConjnthDegreeBound}
\begin{split}
[K(f^{-n}(\alpha)):K] &  \leq  [K(g^{-n}(\beta_1))K(g^{-n}(\beta_2)):K] \\
	&  \leq  [K(g^{-n}(\beta_1)):K][K(g^{-n}(\beta_2)):K].
\end{split}
\end{equation}
and therefore 
\begin{equation}\label{SemiConjLimsupBoundCheb}
\limsup_{n\to+\infty}\frac{1}{d^n}\log[K(f^{-n}(\alpha)):K] \leq \sum_{i=1}^{2}\limsup_{n\to+\infty}\frac{1}{d^n}\log[K(g^{-n}(\beta_i)):K].
\end{equation}
The right-hand-side of (\ref{SemiConjLimsupBoundCheb}) is zero by Proposition \ref{MainArbHausdorffDimCalc}, because we have already proved $\dim(G_{g,\beta_i})=0$ for the power map $g(x)=x^d$.  Therefore the left-hand-side of (\ref{SemiConjLimsupBoundCheb}) vanishes, which implies that $\dim(G_{f,\alpha})=0$ by Proposition \ref{MainArbHausdorffDimCalc}.

\medskip
\noindent
\underline{Case 2: $f(x)=-T_{d}(x)$ and $d$ is even.}  Since $d$ is even, $T_d(x)$ is an even polynomial (\cite{MR2316407} Prop. 6.6 (c)), and hence $f(x)=-T_d(x)=-T_d(-x)$; in other words $f(x)$ is conjugate to $T_d(x)$ by the automorphism $\varphi(x)=-x$.  We conclude that $\udim(G_{f,\alpha})=0$ using Proposition \ref{ConjugacyInvarianceProp} together with the fact that we have already proved the corresponding result for $T_d(x)$. 

\medskip
\noindent
\underline{Case 3: $f(x)=-T_{d}(x)$ and $d$ is odd.}  Since $d$ is odd, $T_d(x)$ is an odd polynomial (\cite{MR2316407} Prop. 6.6 (c)).  Using this fact, together with the property $T_{de}(x)=T_d(T_e(x))$ satisfied by Chebyshev polynomials (\cite{MR2316407} Prop. 6.6 (b)), we have
\[
f^2(x)=-T_d(-T_d(x))=(-1)(-1)T_d(T_d(x))=T_{d^2}(x).
\]
Extending the base field $K$ if necessary, the hypotheses of Lemma \ref{IterateDimPropAlternate} are satisfied with $m=2$ because we have already proved the vanishing of $\udim(G_{f^2,\beta})=\udim(G_{T_{d^2},\beta})$.  Therefore $\udim(G_{f,\alpha})=0$ by Lemma \ref{IterateDimPropAlternate}, completing the proof.
\end{proof}

A rational map $f:\PP^1\to\PP^1$ of degree $d\geq2$ defined over a field $K$ is a {\em Latt\`es map} if there exists a complete nonsingular curve $X/\Kbar$ of genus 1, a morphism $g:X\to X$ of degree $d$ defined over $\Kbar$, and a finite morphism $\varphi:X\to\PP^1$ defined over $\Kbar$ such that $f\circ\varphi=\varphi\circ g$.

\begin{equation*}
\begin{CD}
X @> g >> X \\ 
@V \varphi VV                                    @VV \varphi V \\ 
\PP^1  @> f >>  \PP^1
\end{CD}
\hskip2cm  f\circ\varphi=\varphi\circ g
\end{equation*}

\begin{prop}\label{LattesMapsAreSmall}
Let $(f,\alpha)$ be an arboreal pair of degree $d\geq2$ over a field $K$.  If $f$ is a Latt\`es map, then $\dim(G_{f,\alpha})=0$.
\end{prop}

\begin{proof}
Extending the base field via Proposition \ref{BaseExtensionInvarianceProp}, we can assume that $X,g,\varphi$ are defined over $K$ and that $\varphi^{-1}(\alpha)\subseteq X(K)$.  Enumerate $\varphi^{-1}(\alpha)=\{\beta_1,\beta_2,\dots,\beta_r\}$.  Then we have
\begin{equation*}
\limsup_{n\to+\infty}\frac{1}{d^n}\log[K(f^{-n}(\alpha)):K] \leq \sum_{i=1}^{r}\limsup_{n\to+\infty}\frac{1}{d^n}\log[K(g^{-n}(\beta_i)):K]
\end{equation*}
by precisely the same argument, mutatis mutandis, as the one establishing (\ref{SemiConjLimsupBoundCheb}) in the Chebyshev case.  So by Proposition \ref{MainArbHausdorffDimCalc}, in order to show that $\dim(G_{f,\alpha})=0$, it suffices to show, for each $\beta\in X(K)$, that
\begin{equation}\label{GenusOneLimitCalc}
\lim_{n\to+\infty}\frac{\log[K(g^{-n}(\beta)):K]}{d^n}=0.
\end{equation}

Since $X$ is a genus 1 curve we may view it as an elliptic curve with origin $O=\beta$.  Let $T=g(O)$ and define
\[
h:X\to X \hskip1cm  h(P):=g(P)-T.
\]
Since $h(O)=O$, it follows that $h:X\to X$ is an endomorphism of $X$ of degree $d$ defined over $K$. In particular, it holds that
\begin{equation}\label{IsogenyisHomom}
h(P+Q)=h(P)+h(Q) 
\end{equation}
for all $P,Q\in X(\Kbar)$ (\cite{silverman:aec} Thm. III.4.8).  For each $P\in X(\Kbar)$ we have $g(P)=h(P)+T$, and using (\ref{IsogenyisHomom}) we can recursively calculate more generally that
\begin{equation}\label{GenusOneRecursion}
g^n(P) = h^n(P) + h^{n-1}(T)+h^{n-2}(T)+\dots +T.
\end{equation}

Fix $n\geq0$ and select a point $P_0\in g^{-n}(O)$.  Then for arbitrary $P\in g^{-n}(O)$ we have $g^n(P)=O=g^n(P_0)$, and hence using (\ref{GenusOneRecursion}) we obtain $h^n(P)=h^n(P_0)$.  Since $h$ is an endomorphism, this implies that 
\begin{equation}\label{gDivPoint}
h^n(P-P_0)=O.
\end{equation}

Next denote by $\hat{h}:X\to X$ the dual endomorphism of $h$, characterized by $\hat{h}\circ h=h\circ\hat{h}=[d]$.  (For each $m\in\ZZ$, recall the standard notation $[m]:X\to X$ for the multiplication by $m$ endomorphism.)  Since $h$ and $\hat{h}$ commute with each other we have $\hat{h}^n\circ h^n=[d]^n=[d^n]$ for all $n\geq1$.  Applying $\hat{h}^n$ to both sides of (\ref{gDivPoint}) we obtain
\begin{equation*}
[d^n](P-P_0)=\hat{h}^n\circ h^n(P-P_0)=\hat{h}^n(O)=O.
\end{equation*}
We have shown that every $P\in g^{-n}(O)$ satisfies $P=Z+P_0$ for some $d^n$-torsion point $Z\in X(\Kbar)$.  It follows that
\begin{equation}\label{GenusOneFieldInclusion}
K(g^{-n}(O)) \subseteq K(X[d^n],P_0)
\end{equation}
where $X[d^n]$ is the $d^n$-torsion subgroup of $X(\Kbar)$.  

We have degree bounds 
\begin{equation}\label{GenusOneDegreeBound1}
\begin{split}
[K(P_0):K] & \leq d^n \\
[K(X[d^n]):K] & \leq d^{4n}.
\end{split}
\end{equation}
The first bound in (\ref{GenusOneDegreeBound1}) follows from the fact that $f^n$ is a degree $d^n$ morphism and $f^n(P_0)\in X(K)$.  The second bound in (\ref{GenusOneDegreeBound1}) is a consequence of the faithful action of $\Gal(K(X[d^n])/K)$ on the torsion group $X[d^n]$.  Since $X[d^n]\simeq (\ZZ/d^n\ZZ)\times(\ZZ/d^n\ZZ)$ (recall that $\ch(K)$ is either 0 or $>d$) we obtain an injective group homomorphism
\[
\Gal(K(X[d^n])/K)\hookrightarrow \Aut(X[d^n])\simeq\GL_2(\ZZ/d^n\ZZ)
\]
and the second bound in (\ref{GenusOneDegreeBound1}) follows from $|\GL_2(\ZZ/d^n\ZZ)|\leq d^{4n}$.

We conclude using (\ref{GenusOneFieldInclusion}) and (\ref{GenusOneDegreeBound1}) that
\begin{equation}\label{GenusOneDegreeBound2}
\begin{split}
[K(g^{-n}(O)):K] & \leq [K(X[d^n],P_0):K] \\
	& \leq [K(X[d^n]):K][K(P_0):K] \\
	& \leq d^{5n}.
\end{split}
\end{equation}
From (\ref{GenusOneDegreeBound2}) we obtain the limit (\ref{GenusOneLimitCalc}), completing the proof.
\end{proof}

\begin{proof}[Proof of Theorem \ref{NewConjImpliesAP}]

Assume that Conjecture \ref{MinkowskiDimSmallConj} is true for polynomials.  Let $K$ be a number field, let $f:\PP^1\to\PP^1$ be a polynomial map of degree $d\geq2$ defined over $K$, and let $\alpha\in \PP^1(K)$.  Assume that $G_{f,\alpha}$ is abelian.  We must show that the pair $(f,\alpha)$ is $K^\ab$-conjugate to one of the pairs described in Conjecture \ref{AndrewsPetscheConj}.

Since $G_{f,\alpha}$ is abelian, by Theorem \ref{AbelianSubgroupsAreSmallIntro} we have $\dim(G_{f,\alpha})=0$, so by the statement of Conjecture \ref{MinkowskiDimSmallConj} it holds that the pair $(f,\alpha)$ is $\Kbar$-conjugate to a pair $(g,\beta)$, where either $g(x)=x^d$ or $g(x)=\pm T_d(x)$, and $\beta\in\PP^1(\Kbar)$.  (The cases in which $g(x)$ is either $x^{-d}$ or a Latt\`es map are excluded, as these maps are not $\Kbar$-conjugate to polynomials.)  Let $\varphi\in\PGL_2(\Kbar)$ be the automorphism conjugating the pair $(f,\alpha)$ to the pair $(g,\beta)$.

In the case $g(x)=x^d$, it was proved in \cite{MR4150256} Thm. 12 that, because $G_{f,\alpha}$ is abelian, it must hold that $\beta$ is a root of unity and that the automorphism $\varphi$ must be defined over $K^\ab$, completing the proof in this case.  In the case $g(x)=T_d(x)$, it was proved in \cite{MR4150256} Thm. 13 that, because $G_{f,\alpha}$ is abelian, it must hold that $\beta=\zeta+\frac{1}{\zeta}$ for a root of unity $\zeta$, and that the automorphism $\varphi$ must be defined over $K^\ab$, completing the proof in this case.

It remains to consider the case $g(x)=-T_d(x)$.  When $d$ is even, it holds that $T_d(x)$ is an even polynomial (\cite{MR2316407} Prop. 6.6 (c)).  So $-T_d(x)=-T_d(-x)$, in other words $-T_d(x)$ and $T_d(x)$ are $K$-conjugate to each other by the automorphism $\psi(x)=-x$.  Since the pair $(f,\alpha)$ is $\Kbar$-conjugate to the pair $(g,\beta)$, which is in turn $K$-conjugate to the pair $(-g,-\beta)=(T_d(x),-\beta)$, we deduce again by \cite{MR4150256} Thm. 13 that $-\beta=\zeta+\frac{1}{\zeta}$ for a root of unity $\zeta$ and that the automorphism conjugating $(f,\alpha)$ to $(T_d(x),-\beta)$ is defined over $K^\ab$.  Then $\beta=-\zeta+\frac{1}{-\zeta}$ for the root of unity $-\zeta$, and the automorphism conjugating $(f,\alpha)$ to $(-T_d(x),\beta)$ is also defined over $K^\ab$, as the two automorphisms differ by a factor of $\psi(x)=-x$.

Finally, we consider the case that $g(x)=-T_d(x)$ and $d$ is odd.  In this case $T_d(x)$ is an odd polynomial (\cite{MR2316407} Prop. 6.6 (c)).  Using this fact, together with the property $T_{de}(x)=T_d(T_e(x))$ satisfied by Chebyshev polynomials (\cite{MR2316407} Prop. 6.6 (b)), we have
\[
g^2(x)=-T_d(-T_d(x))=(-1)(-1)T_d(T_d(x))=T_{d^2}(x).
\]
Since $K(f^{-\infty}(\alpha))=K((f^2)^{-\infty}(\alpha))$, we have that $K((f^2)^{-\infty}(\alpha))/K$ is an abelian extension, and thus $G_{f^2,\alpha}$ is abelian.  Recall that $\varphi\in\PGL_2(\Kbar)$ denotes the automorphism conjugating the pair $(f,\alpha)$ to the pair $(g,\beta)$.  Thus $\varphi$ also conjugates the pair $(f^2,\alpha)$ to the pair $(g^2,\beta)=(T_{d^2},\beta)$.  We conclude using the previous case that $\beta=\zeta+\frac{1}{\zeta}$ for a root of unity $\zeta$, and that $\varphi$ is defined over $K^\ab$, completing the proof.
\end{proof}

\def\cprime{$'$}



\end{document}